\documentclass[11pt,reqno,a4paper]{amsart}

\usepackage{amsmath}
\usepackage{enumerate}
\usepackage{amssymb}
\usepackage{amscd}
\usepackage{amsthm}
\usepackage{amsfonts}
\usepackage{graphicx}
\usepackage{hyperref}

\usepackage{fouriernc}
\usepackage[T1]{fontenc}

\numberwithin{equation}{section}

\newcommand{\SL}{\operatorname{SL}}

\newcommand{\GL}{\operatorname{GL}}
\newcommand{\PSL}{\operatorname{PSL}}
\newcommand{\Aut}{\operatorname{Aut}}

\newcommand{\cH}{\mathcal{H}}

\newcommand{\bC}{\mathbb{C}}

\newcommand{\bK}{\mathbb{K}}

\newcommand{\bN}{\mathbb{N}}

\newcommand{\bQ}{\mathbb{Q}}
\newcommand{\bR}{\mathbb{R}}

\newcommand{\bT}{\mathbb{T}}

\newcommand{\bZ}{\mathbb{Z}}

\newcommand{\ra}{\rightarrow}

\newcommand{\qand}{\quad \textrm{and} \quad}

\newcommand\subsetsim{\mathrel{%
\ooalign{\raise0.2ex\hbox{$\subset$}\cr\hidewidth\raise-0.8ex\hbox{\scalebox{0.9}{$\sim$}}\hidewidth\cr}}}
\newcommand{\eps}{\varepsilon}

\DeclareMathOperator{\supp}{supp}

\DeclareMathOperator{\cobnd}{Cobnd}

\DeclareMathOperator{\linspan}{span}

\DeclareMathOperator{\Sym}{Sym}

\DeclareMathOperator{\Ind}{Ind}
\DeclareMathOperator{\FC}{FC}

\theoremstyle{theorem}
\newtheorem{theorem}{Theorem}[section]
\newtheorem{corollary}{Corollary}[section]
\newtheorem{proposition}{Proposition}[section]
\newtheorem{lemma}{Lemma}[section]

\theoremstyle{definition}
\newtheorem{definition}{Definition}[section]
\newtheorem{remark}{Remark}[section]
\newtheorem*{example}{Example}
\newtheorem*{question}{Question}

\begin{document}

\title{Ergodic theorems for coset spaces}

\author{Michael Bj\"orklund}
\address{Department of Mathematics, Chalmers, Gothenburg, Sweden}
\email{micbjo@chalmers.se}
\thanks{}

\author{Alexander Fish}
\address{School of Mathematics and Statistics, University of Sydney, Australia}
\curraddr{}
\email{alexander.fish@sydney.edu.au}
\thanks{}



\date{}

\dedicatory{}

\begin{abstract} 
We study in this paper the validity of the mean ergodic theorem along 
\emph{left} F\o lner sequences in a countable amenable group $G$. 
Although the \emph{weak} ergodic theorem always holds along 
\emph{any} left F\o lner sequence in $G$, we provide examples 
where the \emph{mean} ergodic theorem fails in 
quite dramatic ways. On the other hand, if $G$ does 
not admit any ICC quotients, e.g. if $G$ is virtually nilpotent, then 
we prove that 
the mean ergodic theorem does indeed hold along \emph{any} left F\o lner 
sequence. 

In the case when a unitary representation of a countable amenable group 
is induced from a unitary representation of a "sufficiently thin" subgroup
, we prove that the mean ergodic theorem holds along any left F\o lner sequence for this representation. 

Furthermore, we show that every countable (infinite) amenable group 
$L$ embeds into a countable group $G$ which admits a unitary representation with the property that for any left 
F\o lner sequence $(F_n)$ in $L$, there exists a sequence $(s_n)$ in 
$G$ such that the mean (but \emph{not} the weak) ergodic theorem fails for this representation along the sequence $(F_n s_n)$.

Finally, we provide examples of countable (not necessarily amenable)
groups $G$ with proper, infinite-index subgroups $H$, so that the 
\emph{pointwise} ergodic theorem holds for averages 
along \emph{any} strictly increasing and nested sequence of finite 
subsets of the coset $G/H$.
\end{abstract}

\maketitle

\section{Introduction}

\subsection{General comments}
We recall that the celebrated strong (mean) ergodic theorem of J. von Neumann asserts that whenever $U$ is a unitary operator on a 
Hilbert space $\cH$, then all of the limits
\[
\lim_n \frac{1}{n} \sum_{k=0}^{n-1} U^k v = P_U v 
\]
exist in the norm topology on $\cH$ for every $v \in \cH$, where 
$P_U$ denotes the orthogonal projection onto the closed linear 
subspace $\cH^U \subset \cH$ spanned by all $U$-invariant vectors. 
More generally, we can consider a unitary representation $(\cH,\pi)$ of
a countable group $G$ and ask for conditions on sequences $(F_n)$
of finite subsets of $G$ which are strong enough to ensure that the 
limits
\begin{equation}
\label{met}
\lim_n \frac{1}{|F_n|} \sum_{s \in F_n} \pi(s)v = P_G v
\end{equation}
exist in the norm (or weak) topology on $\cH$ for every $v \in \cH$, where 
$P_G$ denotes the orthogonal projection onto the closed linear 
subspace $\cH^G \subset \cH$ spanned by all $\pi(G)$-invariant 
vectors. We shall say that such sequences are \emph{strong} for
$(\cH,\pi)$, or 
\emph{weak} if the limits are only guaranteed to exist in the weak topology on $\cH$. Clearly,
every strong sequence is a weak sequence for any unitary representation. \\

As was pointed out by W. F. Eberlein in \cite{Eb49} (based on some 
fundamental observations by F. Riesz in \cite{R45}), if $G$ is a 
countable \emph{amenable} group, then every \emph{right} 
F\o lner sequence in $G$ is a strong sequence for every unitary
representation of $G$. We recall that a 
sequence $(F_n)$ of finite subsets of $G$ is a 
\emph{right F\o lner sequence} if
\[
\lim_n \frac{|F_n \triangle F_n t|}{|F_n|} = 0, \quad \textrm{for all $t \in G$},
\]
and it is a \emph{left F\o lner sequence} if 
\[
\lim_n \frac{|F_n \triangle t F_n|}{|F_n|} = 0, \quad \textrm{for all $t \in G$},
\]
and we say that a countable group is \emph{amenable} if it admits a 
right F\o lner sequence. Since the inverse of any right F\o lner sequence
is a left F\o lner sequence, we could just as well have required in the definition of an amenable group that it should admit a left F\o lner 
sequence.\\

The class of amenable groups is quite large and contains all solvable (hence all
abelian) groups, all locally finite groups and all finitely generated groups
of subexponential growth. However, free groups on two or more 
generators and countably infinite groups with property (T) are well-known examples of non-amenable groups. \\

If $G$ is abelian (or more generally, a finite extension of an abelian group),
then every left F\o lner sequence in $G$ is also a right F\o lner sequence. 
However, this property is very special for this class of groups, and we note that if $G$
has at least \emph{one} element with an \emph{infinite} conjugation class, then there is a
left F\o lner sequence which is very far from being a right F\o lner sequence;
indeed, let $(F_n)$ be any left F\o lner sequence in such a group and suppose
that $t$ is an element in $G$ with an infinite conjugation class. Then we can find a sequence
$(s_n)$ in $G$ such that the sets $F_n$ and $F_ns_n t s_n^{-1}$ are disjoint 
for every index $n$, and thus the left F\o lner sequence $(F_n s_n)$ satisfies 
\[
\frac{|F_n s_n \triangle F_n s_n t|}{|F_n|} 
=
\frac{|F_n \triangle F_n s_n t s_n^{-1}|}{|F_n|} = 2 
\quad \textrm{for all $n$},
\]
which shows that $(F_ns_n)$ is \emph{not} a right F\o lner sequence.  \\

It is not hard to prove (and we shall give the proof 
in Section \ref{sec:lwet}) that left F\o lner sequences in any 
countable amenable group $G$ are \emph{weak} sequences for 
any unitary representation of $G$. One of the main aims of this
paper is to investigate to which extent left F\o lner sequences in 
countable amenable groups are also \emph{strong} sequences (at 
least for large classes of unitary representations).

\subsection{Left amenable triples}

We shall study a slightly more general setting than the one described
in the previous subsection. We begin by introducing some of the
central concepts needed to state and explain the main results of this paper.

\begin{definition}[Strong and weak sequences]
\label{def: averaging seq}
Let $G$ be a countable group and let $H$ and $L$ be subgroups of $G$.
We shall say that a sequence $(F_n)$ of finite subsets of $G/H$ is a 
\emph{strong (weak) sequence for the triple $(G,H,L)$} if the limits 
\begin{equation}
\label{def. averaging seq}
\lim_n \, \frac{1}{|F_n|} \sum_{s \in F_n} \pi(s)v = P_L v
\end{equation}
exist in the norm (weak) topology on $\cH$ for all $v \in \cH^H$ 
(the closed linear subspace of $\cH$ spanned by $\pi(H)$-invariant vectors).
\end{definition}

We refer the reader to Section \ref{examples} for a list of examples of triples
of countable groups, some of which have been extensively studied in the literature. \\

Given a triple $(G,H,L)$ as in the definition above, we note that $G$, and hence the subgroup $L$, always admits a \emph{left} action on the coset space $G/H$, and with respect to this action, the following extension of the notion of a left F\o lner sequence seems natural.

\begin{definition}[Left F\o lner sequence]
We say that a sequence $(F_n)$ of finite subsets of $G/H$ is a 
\emph{left F\o lner sequence for $(G,H,L)$} if
\begin{equation}
\label{def. right folner}
\lim_n 
\frac{\big| \, F_n \, \triangle \, s F_n \, \big|}{\big|F_n\big|}
= 0, \quad \textrm{for all $s$ in $L$},
\end{equation}
and we say that the triple $(G,H,L)$ is \emph{left amenable} if it 
admits a left F\o lner sequence.
\end{definition}

\begin{remark}
We stress that the notion of left amenable triples is not a new, although
in the literature it is more commonplace to write that \emph{$L$ acts amenably on $G/H$}. Also, in the special case when the triple $(G,H,G)$ is left amenable, one often says that $H$ is \emph{co-amenable} in $G$. It is not
hard to show that if $G$ is an amenable group, then $(G,H,L)$ is left amenable for any pair of subgroups $H$ and $L$ of $G$.
\end{remark}

The proof of the following theorem is not hard, but for completeness we 
have decided to give the short proof in Section \ref{sec:lwet}.

\begin{theorem}[Left Weak Ergodic Theorem]
\label{lwet}
Every left F\o lner sequence for a left amenable triple is a \emph{weak}
sequence for the same triple.
\end{theorem}

We recall that if $(\cH,\pi)$ is a unitary representation of a countable 
group $G$, then we can decompose $\cH$ into an orthogonal direct
sum of two sub-representations $\cH_{\textrm{fin}}$ and 
$\cH_{\textrm{wm}}$, where $\cH_{\textrm{fin}}$ denotes the closed
linear span of all of those vectors in $\cH$ whose cyclic linear spans are 
\emph{finite-dimensional}, and thus its orthogonal complement 
$\cH_{\textrm{wm}}$ in $\cH$ has no finite-dimensional sub-representations. We say that the unitary representation 
$(\cH_{\textrm{fin}},\pi|_{\cH_{\textrm{fin}}})$ has \emph{pure point
spectrum} and we say that $(\cH_{\textrm{wm}},\pi|_{\cH_{\textrm{wm}}})$ 
is \emph{weakly mixing}. \\

Since 
the weak topology and the norm topology coincide on finite-dimensional Hilbert spaces, we have the following corollary of Theorem \ref{lwet}.

\begin{corollary}
If $G$ is a countable amenable group, then every left F\o lner sequence
in $G$ is a strong sequence for any unitary representation with pure point spectrum.
\end{corollary}

We note that if we want to show that a given sequence $(F_n)$ of finite
subsets of a countable group $G$ is a strong sequence for \emph{any}
unitary representation, then it suffices to show this for 
\emph{irreducible} unitary representations (see e.g. Proposition 2.3.2. in \cite{Zi}). In particular, by the previous
corollary, we conclude that left F\o lner sequences are always strong for the class of countable groups (so called \emph{Moore groups}) all of whose irreducible unitary representations are finite-dimensional.

\begin{corollary}[von Neumann's Ergodic Theorem]
If $G$ is a countable Moore group, then every left F\o lner sequence in 
$G$ is a strong sequence for any unitary representation.
\end{corollary}

Every countable abelian group is a Moore group, and 
conversely, any countable Moore group is a finite extension of a 
countable abelian group, which was proved by  E. Thoma in \cite{T64}.

\subsection{Right amenable triples}
We note here that in some cases of interest, it is also possible to define a natural notion of a right F\o lner sequence for a triple. In order to explain
this notion, let $N_G(H)$ denote the normalizer of $H$ in $G$, and observe that $N_G(H)$ admits an action on the coset space $G/H$ on the right via 
the map
\[
c_H(s)xH = xs^{-1}H \quad \textrm{for $s \in N_G(H)$}.
\]
This $N_G(H)$-action commutes with the left action of $G$ on the coset space $G/H$ and factors through the quotient $H \backslash N_G(H)$. 
We can now extend the notion of a right F\o lner sequence to the setting of triples as follows.

\begin{definition}[Right F\o lner sequence]
Suppose that the inclusions $H < L < N_G(H)$ hold. We say that a sequence $(F_n)$ of 
finite subsets of $L/H$ is a \emph{right F\o lner sequence for 
$(G,H,L)$} if 
\begin{equation}
\label{def. right folner}
\lim_n 
\frac{\big| \, F_n \, \triangle \, c_H(s) F_n \, \big|}{\big|F_n\big|}
= 0, \quad \textrm{for all $s$ in $L$.}
\end{equation}
We say that the triple $(G,H,L)$ is \emph{right amenable} if it admits
a right F\o lner sequence $(F_n)$ such that $|F_n| \ra  \infty$. In 
the special 
case when $L = G$ and $H$ is the trivial subgroup, this
notion coincides with the definition of a right F\o lner sequence in $G$
given above. 
\end{definition}

\begin{remark}
If the quotient group $L/H$ is an \emph{infinite} amenable 
group and $(K_n)$ is a right F\o lner sequence in $L/H$, then 
$(c_H(K_n^{-1})H)$ is a right F\o lner sequence for $(G,H,L)$. We stress 
that the notion of a right F\o lner sequence in this generality is much more restrictive than the corresponding notion of a left F\o lner sequence. For
instance, since right F\o lner sequences are assumed to be contained in 
the coset space $L/H$, the notion of right amenability (existence of a 
right F\o lner sequence whose sizes tend to infinite) only makes sense 
if $L/H$ itself is infinite. In particular,
right amenability is essentially a void concept if $H$ is self-normalized in $G$, that is to say, if $N_G(H) = H$.
\end{remark}

In Section \ref{sec:rset} we shall adapt a classical argument of F. Riesz to
the setting of triples and establish the following \emph{strong} ergodic 
theorem.

\begin{theorem}[Right Strong Ergodic Theorem]
\label{rset}
Every right F\o lner sequence for a right amenable triple is 
a \emph{strong} sequence for the same triple.
\end{theorem}

An unfortunate feature with right F\o lner sequences for a right 
amenable triple $(G,H,L)$ is that one cannot really hope to gain 
more than $L$-invariance upon averaging unitary representations 
with $\pi(H)$-invariant vectors. In particular, if the normalizer of 
$H$ in $G$ is small, then this may not be very useful. 

\subsection{A transfer principle}
We shall now return to our discussion concerning the validity of the 
mean ergodic theorem along \emph{left} F\o lner sequences. \\

Let $(G,H,L)$ be a left amenable triple and let $(\cH,\pi)$ be a 
unitary representation of $G$ with non-zero $\pi(H)$-invariant
vectors. We say
that the \emph{Left Strong Ergodic Theorem holds for $(\cH,\pi)$}
if every left F\o lner sequence for $(G,H,L)$ is strong for $(\cH,\pi)$,
that is to say, if the limits
\[
\lim_n \frac{1}{|F_n|} \sum_{s \in F_n} \pi(s)v = P_L v
\]
exists in the norm topology for every $v \in \cH^H$ and for \emph{every} left 
F\o lner sequence for $(G,H,L)$. If this is the case for \emph{every} unitary representation of $G$, then we simply say that the Strong Left Ergodic
Theorem holds for the triple $(G,H,L)$, and in the special case when 
$G = L$ and $H$ is the trivial subgroup, we say that the Strong Left
Ergodic Theorem holds for $G$. \\

Our first main contribution to this paper is the observation that
the question whether the Left Strong Ergodic Theorem holds for a given
countable \emph{amenable} group can be equivalently phrased in a way
which does not use left F\o lner sequences (and hence not amenability).
The key notion behind this reformulation is that of a \emph{firm} strong
sequence, which we define as follows.

\begin{definition}[Firm strong sequence]
Let $(\cH,\pi)$ be a unitary representation of a countable (not necessarily amenable) group $G$. 
A sequence $(K_m)$ of finite subsets of $G$ is called a \emph{firm 
strong sequence with respect to $(\cH,\pi)$} if for every choice of a sequence 
$(t_m)$ in $G$, we have
\[
\lim_n \frac{1}{|K_m|} \sum_{t \in K_m} \pi(tt_m)v = P_G v,
\]
for all $v \in \cH$ in the norm topology on $\cH$. If $(K_m)$ is a firm strong sequence for \emph{every} unitary representation of $G$, then we shall simply say that $(K_m)$ is 
a \emph{strong firm sequence in $G$}.
\end{definition}

We note that if $G$ is a countable \emph{abelian} group, then every strong
sequence in $G$ is firm, but as we shall see below, many countable 
(amenable and non-amenable) groups do not admit firm sequences. 
In fact, we have the following equivalence result.

\begin{proposition}[Transfer principle]
\label{transfer0}
The Strong Left Ergodic Theorem holds for a countable amenable group
if and only if the group admits a firm strong sequence.
\end{proposition}

In particular, this theorem allows us to "localize" the question whether the
Left Strong Ergodic Theorem holds along \emph{any} left F\o lner sequence
to whether it holds along all right translates of a \emph{fixed} sequence of 
finite subsets of the group (not necessarily a left F\o lner sequence).

\subsection{The Left Strong Ergodic Theorem for countable nilpotent groups}

We recall that a countable group $G$ is \textrm{ICC} (Infinite Conjugation Classes) if every non-trivial element in $G$ has an infinite conjugation 
class. For instance, free groups on two or more generators are ICC, so 
are lattices in connected simple real Lie groups (a direct consequence of
the Borel Density Theorem, see e.g. Proposition 2 in \cite{BLH}) and many countable non-nilpotent solvable groups. 

On the other hand, countable nilpotent groups have non-trivial centers and 
are thus never ICC. More generally, countable groups which contain finite-index nilpotent subgroups (so called \emph{virtually nilpotent} groups) are never ICC. \\

We can now state our first main result in this paper as follows.

\begin{theorem}
\label{lset}
If $G$ is a countable amenable group without ICC quotients, then every 
\emph{right} F\o lner sequence in $G$ is a firm strong sequence. In particular, 
the Strong \emph{Left} Ergodic Theorem holds for every countable virtually nilpotent group.
\end{theorem}

\begin{remark}
This result might seem a bit suspicious since there is absolutely no reason
to expect that right translates of any right F\o lner sequence in $G$ should
exhibit any asymptotic invariance whatsoever (either on the left or on the right hand side). We stress that the existence of \emph{finite} conjugation 
classes in every quotient group of $G$ is the mechanism which will drive 
the argument. 
\end{remark}

The arguments in the proof of Theorem \ref{lset} will also yield the following
corollary, and will be outlined in Subsection \ref{subsec:redicc}.

\begin{corollary}[Reduction to ICC groups]
\label{redicc}
Let $G$ be a countable amenable group for which the Left Strong Ergodic 
Theorem fails. Then $G$ admits a quotient group which is ICC and 
for which the Left Strong Ergodic Theorem fails.
\end{corollary}

We shall later provide an example of a countable (infinitely generated) 
ICC group for which the Left Strong Ergodic Theorem fails in a very 
dramatic way. On the other hand, we have not been able to construct a 
\emph{single} example of a countable amenable ICC group for which the 
Left Strong Ergodic Theorem holds for \emph{all} unitary representations. 
Therefore it seems natural to ask the following question.

\begin{question}
Does the Left Strong Ergodic Theorem fail for \emph{every} countable 
amenable ICC group?
\end{question}

There are at least three reasons why this is an intricate question. The first 
reason is that we do \emph{not} know at the moment whether there are
\emph{countable} (amenable) groups for which \emph{every} unitary representation 
is mixing. Recall that a unitary representation $(\cH,\pi)$ of a 
countable group $G$ is \emph{mixing} if for every $v \in \cH$ and
$\eps > 0$, there exists a finite set $F \subset G$ such that 
\[
\sup_{s \in G \setminus F} \big| \langle v, \pi(s)v \rangle \big| < \eps.
\]
One can readily check that every sequence of finite sets in $G$ whose 
sizes tend to infinity is a strong sequence for any mixing unitary representation of $G$, and thus the Left Strong Ergodic Theorem for 
these representations is immediate.  The question whether  
countable (amenable) groups, all of whose unitary representations 
are mixing, exist was first raised by K. Schmidt in the paper \cite{Kmix}, where it is 
proved that such groups must be finitely generated and can never 
admit proper \emph{infinite} subgroups. In particular, possibly 
modulo a finite normal subgroup, such a group is ICC,
but the Left Strong Ergodic Theorem would nevertheless hold
for such a group. \\

The second reason why the question above is intricate has to do with
the fact that there are many \emph{non-amenable} ICC groups 
which admit firm strong sequences as the following proposition 
shows.

\begin{proposition}[Gorodnik-Nevo, Theorem 1.7 in \cite{GN}]
\label{gn}
If $\tilde{G}$ is a non-compact and connected simple Lie group with 
trivial center and real rank at least two, e.g. $\PSL_n(\bR)$ for 
$n \geq 3$, then every lattice in $\tilde{G}$ is ICC and 
admits a firm strong sequence. 
\end{proposition}

We stress that this is not how Theorem 1.7 in \cite{GN} is stated and
for the reader's convenience we sketch in the appendix of this paper 
the derivation of the formulation above from that of Theorem 1.7 in 
\cite{GN}. The proof that lattices in non-compact and connected 
simple Lie groups with trivial centers are ICC (for this, the real rank
assumption is unnecessary) utilizes the Borel Density
Theorem and can be found in the paper \cite{BLH} (Proposition 2). 

\subsection{A strong ergodic theorem for induced representations}
The third reason why the question in the previous subsection is intricate 
stems from our very poor understanding of \emph{irreducible} unitary
representations of a general countable (discrete) group, which is not 
a finite extension of an abelian group. \\

In fact, for very large classes of countable groups, the only arsenal of unitary representations (irreducible or not) which we have at our disposal are
the ones which are \emph{induced} from various subgroups, whose unitary 
representation theory we know better. \\

We recall that
if $G$ is a countable group and $L$ is a proper subgroup of $G$, then 
a unitary representation $(\cH,\pi)$ of $G$ is \emph{induced} from $L$,
if there exists a unitary representation $(\cH,\pi_o)$ of $L$ such that 
$(\cH,\pi)$ is isomorphic to the right regular representation on the 
(pre-)Hilbert space 
\[
\Ind_L^G \pi_o = 
\Big\{ 
f : G \ra  \cH_o \, : \, \|f\| < \infty \qand f(lg) = \pi_o(l)^{-1}f(g), 
\quad 
\textrm{for all $l \in L$ and $g \in G$}
\Big\}.
\]
Here, $\|\cdot\|$ denotes the norm induced from the following 
inner product 
\[
\langle f,f' \rangle = \sum_{x \in L \backslash G} \langle f(x), f'(x) \rangle_o,
\]
where $\langle \cdot, \cdot \rangle_o$ is the inner product on $\cH_o$. 
In particular, we note that $\Ind_L^G \textrm{id}_L$ is isomorphic to the left regular
representation of $G$ on $\ell^2(G/L)$. \\

A classical and very useful result of G. Mackey (see e.g. Theorem 6 in \cite{Mackey}) asserts that if $L$ is a subgroup of a countable group $G$ such that 
every non-identity coset in $G/L$ has an infinite $L$-orbit, then for 
\emph{any} \emph{finite-dimensional} irreducible unitary representation 
$(\cH_o,\pi_o)$ of $L$, the induced representation $\Ind_L^G \pi_o$ is 
irreducible. In particular, if $L$ is such a subgroup, then the left regular representation of $G$ on 
$\ell^2(G/L)$ is irreducible. \\

Our main result in this paper pertaining to induced representations,
shows in particular 
that if $G$ is a countable amenable group and $L$ is a "sufficiently thin" subgroup
in $G$, then induced representations from $L$ can never constitute  
counterexamples to the Left Strong Ergodic Theorem for the group $G$.

\begin{proposition}[Ergodic theorem for induced representations]
\label{induced}
Let $G$ be a countable group and let $(F_n)$ be a sequence of finite subsets of $G$. Suppose that $L$ is a subgroup of $G$ which satisfies 
\begin{equation}
\label{thin subgroup}
\lim_n \sup_{x,y \in G/L} \frac{|F_n \cap y Lx^{-1}|}{|F_n|} = 0.
\end{equation}
If $(\cH_o,\pi_o)$ is a unitary representation of $L$, then $(F_n)$ is 
a firm strong sequence with respect to the unitary induction
$\Ind_{L}^G \pi_o$.
\end{proposition}

\begin{remark}
We stress that we do \emph{not} assume that the Left Strong Ergodic Theorem holds for the unitary representation $(\cH_o,\pi_o)$ of the 
subgroup $L$.
\end{remark}

In Section \ref{sec:induced} we shall prove this result and discuss the thinness criterion \eqref{thin subgroup} for subgroups of semidirect products. \\

We stress that Proposition \ref{induced} does \emph{not} imply that induced representations are never counterexamples to the Left Strong Ergodic Theorem. In fact, we shall in connection with Theorem \ref{thmfail} below introduce the notion of a conjugation-thick subgroup  whose
very definition is a strong violation of condition \eqref{thin subgroup}. 
We shall provide an example of a countable (amenable) group $G$ with 
a proper conjugation-thick subgroup $L$ and show that the Left Strong 
Ergodic Theorem fails in a very dramatic way for the left regular 
representation of $G$ on $\ell^2(G/L)$, which is the induction of the 
identity representation on $L$ to $G$.

\subsection{Non-property (T) groups and flabby pairs}

We shall now begin to discuss various failures of the Left Strong 
Ergodic Theorem in a more systematic fashion. For this, we 
need the notions of \emph{flabby pairs} and \emph{flabby groups},
where the latter can be seen as a very strong obstruction to the 
existence of a firm strong sequence in the group.

\begin{definition}[Flabby pair]
Let $G$ be a countable (not necessarily amenable) group and let $L$ be a subgroup of $G$. We 
say that the pair $(G,L)$ is \emph{flabby} if there exists a unitary representation
$(\cH,\pi)$ of $G$ with no non-zero $\pi(L)$-invariant vectors with the
property that for any sequence $(F_n)$ of finite subsets of $L$, whose 
sizes tend to infinity, there 
exists a sequence $(s_n)$ in $G$ such that
\[
\lim_n \Big \|\frac{1}{|F_n|} \sum_{s \in F_n} \pi(ss_n)v \, \Big\| = 1,
\]
for some unit vector $v$ in $\cH$, and we say that $G$ is 
\emph{flabby} if the pair $(G,G)$ is flabby. 
\end{definition}

We shall define and discuss Kazhdan's property (T) in Subsection 
\ref{sec:embed}, where the following embedding theorem is proved. 

\begin{theorem}[Embedding into flabby pairs]
\label{embed flabby}
Every countable group $L$ without Kazhdan's property (T), for instance, 
every countable \emph{infinite} amenable group, embeds into a countable group $G$ such that $(G,L)$ is a flabby pair.
\end{theorem}

In the case when $L$ is a countable (infinite) amenable group, then 
$(G,\{e\},L)$ is always a left amenable triple and $(F_n s_n)$ is a
left F\o lner sequence for this triple whenever $(F_n)$ is a left 
F\o lner sequence in $L$ and $(s_n)$ is any sequence in $G$. 

In particular, by the Left Weak Ergodic Theorem, $(F_n s_n)$ is a 
\emph{weak} sequence for $(G,\{e\},L)$. However, 
as Theorem \ref{embed flabby} demonstrates, it is always 
possible to choose $(s_n)$ so that $(F_n s_n)$ is \emph{not} a 
strong sequence for $(G,\{e\},L)$, so the Left Strong Ergodic 
Theorem fails for this triple.

\subsection{Explicit examples where the Left Strong Ergodic Theorem fails}

Unfortunately, the construction of the embedding in Theorem \ref{embed flabby} is far from explicit, and thus does not yield any insight to the question 
whether the Left Strong Ergodic Theorem can fail for a countable amenable group
(and not just for a left amenable triple). As it turns out, the following 
notion of a \emph{contracting triple} will provide a very convenient and concrete framework for discussing such failures.

\begin{definition}[Contracting triples]
Let $G$ be a countable group and let $H$ and $L$ be subgroups
of $G$. The triple $(G,H,L)$ is called \emph{contracting} if $H < L$ and 
for every finite subset $F \subset L$, there exists $t \in G$ such 
that $tFt^{-1} \subset H$. 

In the case when $G = L$, we say that $H$ is a \emph{conjugation-thick subgroup}.
\end{definition}

We shall discuss various examples of contracting triples and conjugation-thick subgroups in Section \ref{examplesconj}. At this point of the 
exposition we shall simply state and prove the following key result about these triples.

\begin{theorem}[Failure for contracting triples]
\label{thmfail}
If $(G,H,L)$ is a contracting triple such that $H$ has infinite index in $L$, 
then the pair $(G,L)$ is flabby. In particular, the Left Strong Ergodic Theorem
fails for the triple $(G,\{e\},L)$.
\end{theorem}

Indeed, suppose that $(G,H,L)$ is a contracting triple and let $(F_n)$ be any
sequence of finite subsets of $L$. Let $(\cH,\pi)$ denote the left regular 
representation of $G$ on $\ell^2(G/H)$ and let $v_o$ denote the indicator 
function of the identity coset in $G/H$. We note that $v_o$ is a $\pi(H)$-invariant unit vector in $\cH$ and since $H$ has infinite index in $L$ there
are no non-zero $\pi(L)$-invariant vectors in the Hilbert space $\ell^2(G/H)$. By the contraction property of the triple $(G,H,L)$, we can find a sequence $(s_n)$ in $G$ such that the inclusions $s_n^{-1} F_n s_n \subset H$ hold for all $n$ and thus 
\[
\frac{1}{|F_n|} \sum_{s \in F_n} \pi(ss_n)v_o  =\pi(s_n)v_o
\quad 
\textrm{for all $n$},
\]
since $v_o$ is $\pi(H)$-invariant, so in particular the Hilbert norm of this expression equals one for every $n$,
which shows that the pair $(G,L)$ flabby. \\

The following corollary now follows from Theorem \ref{thmfail} and 
the discussions in Section \ref{examplesconj}.

\begin{corollary}
\label{cor:finsym}
The (locally finite) group $\Sym_o(\bN)$ of all permutations on $\bN$ with finite supports is flabby, and thus the Strong Left Ergodic Theorem fails for
this group.
\end{corollary}

\subsection{Automatic pointwise ergodic theorems}

The final section of this paper has as its goal to show that mean and 
pointwise ergodic theory of triples can exhibit dramatically different 
phenomena than in the classically studied "group case". This will be 
done by providing examples of triples for which the mean and 
pointwise ergodic theorem hold along \emph{any} sequence of 
subsets. \\

The key concept here is that of a \emph{rigid pair}, which we define 
as follows.

\begin{definition}[Rigid pairs]
Let $G$ be a countable group and let $H$ be a subgroup of $G$. We say
that the pair $(G,H)$ is \emph{rigid} if whenever $(\cH,\pi)$ is a unitary
representation of $G$ with a no $\pi(G)$-invariant vectors, but equipped
with a $\pi(H)$-invariant unit vector $v_o$ in $\cH$, then 
\[
\langle v_o, \pi(s)v_o \rangle = 
\left\{
\begin{array}{cc}
1 & \textrm{if $s \in H$} \\
0 & \textrm{otherwise}
\end{array}
\right.;
\]
in other words, the cyclic sub-representation spanned by $v_o$ is isomorphic
to the left regular representation of $G$ on $\ell^2(G/H)$. 
\end{definition}

In Section \ref{sec:automatic} we shall show that the pairs
\[
G = \bQ \rtimes \bQ^{*} \qand H = \{0\} \rtimes \bQ^{*}
\]
and
\[
G = \SL_n(\bQ) \times \SL_n(\bQ) \qand H = \Delta_2(\SL_n(\bQ)),
\] 
for $n \geq 2$, where $\Delta_2(\SL_n(\bQ))$ denotes the 
diagonal subgroup in $G$, are rigid. \\

Our main observation concerning rigid pairs, which will be proved in Section \ref{sec:automatic}, can be stated as follows.

\begin{theorem}[Automatic Pointwise Ergodic Theorem]
\label{automatic}
Let $(G,H)$ be a rigid pair and suppose that $(X,\nu)$ is an 
ergodic probability measure-preserving $G$-space and 
$\phi$ is a square-integrable measurable function on $X$ 
which is essentially $H$-invariant. Then, for \emph{every} 
strictly increasing nested sequence $(F_n)$ of finite subsets of 
$G/H$, we have
\[
\lim_n \frac{1}{|F_n|} \sum_{s \in F_n} \varphi(s^{-1}x) = \int_X \varphi \, d\nu
\]
for almost every $x \in X$ with respect to $\nu$.
\end{theorem}

\subsection{Organization of the paper}
The paper is divided into $10$ sections as follows. \\

In Section \ref{examples} we recall the GNS-construction associated to a positive definite function on a countable group and list a few examples 
of triples (mostly of the form $(G,H,G)$ for various countable groups $G$
and subgroups $H$ thereof) which have been studied extensively in the literature and to which our results apply. \\

In Section \ref{sec:lwet} we give a short proof of the Left Weak Ergodic 
Theorem, and in Section \ref{sec:rset} we adapt F. Riesz' classical, and
very elegant, argument to prove the Right Strong Ergodic Theorem to 
our setting. \\

In Section \ref{sec:transfer} we prove that the validity of the Left Strong
Ergodic Theorem for a countable amenable group is equivalent to the
existence of a firm strong sequence. We then apply this result in Section 
\ref{sec:lseticc} to show that the Left Strong Ergodic Theorem holds for 
all countable groups with no ICC quotients. \\

In Section \ref{sec:induced} we prove that counterexamples to the Left 
Strong Ergodic Theorem cannot stem from induced representations from
"thin" subgroups, and we discuss the situation for semidirect products of
groups in detail. \\

In Section \ref{sec:embed} we prove that every countable group $L$
without property (T), e.g. every amenable group, embeds into a 
countable group $G$ so that the Left Strong Ergodic Theorem fails 
for the triple $(G,\{e\},L)$. \\

In Section \ref{examplesconj} we give explicit examples of contracting 
triples and conjugation-thick subgroups. In particular, we show that 
the locally finite group $\Sym_o(\bN)$ contains proper conjugation-thick subgroups, thereby establishing Corollary \ref{cor:finsym}.\\

In Section \ref{sec:automatic} we provide examples of so called rigid 
pairs $(G,H)$ of groups for which pointwise ergodic theorem is valid 
along \emph{any} strictly increasing nested sequence of subsets of the 
coset space $G/H$, and we prove Theorem \ref{automatic}.

\subsection{Acknowledgments}
The authors wish to express their thanks to Michael Cantrell, Alex Furman, Anders Karlsson, Amos Nevo, Felix Pogorzelski and Andreas Thom for interesting and encouraging discussions. The authors would also like to thank Max Planck Institute for Mathematics, Bonn, for their hospitality.

\section{The GNS-construction and unitary representations of coset spaces}
\label{examples}

We shall now give a short list of examples of pairs $(G,H)$ of countable groups,
where $H$ is a subgroup of $G$, whose unitary representations with 
non-zero $\pi(H)$-invariant vectors have attracted some attention in 
the literature. 

\subsection{The GNS-construction}
Let $G$ be a countable group. A complex-valued function $\phi : G \ra \bC$
is called \emph{positive definite} if for every finite set 
$c_1,\ldots,c_N$ of complex numbers and for every finite set 
$x_1,\ldots,x_N$ in $G$, we have
\[
\sum_{i,j = 1}^N c_i \overline{c}_j \, \phi(x_i^{-1}x_j) \geq 0.
\]
We note that if $(\cH,\pi)$ is a unitary representation of $G$ and $v$ is any 
vector in $\cH$, then the function 
\[
\phi_v(x) = \langle v, \pi(x)v \rangle \quad \textrm{for $x \in G$}
\]
is positive definite on $G$. Furthermore, if $v$ is $\pi(H)$-invariant for some subgroup 
$H < G$, then $\phi_v$ is bi-$H$-invariant. \\

Conversely, if $H$ is a subgroup of $G$ and $\phi$ is a positive definite
function on $G$ which is bi-$H$-invariant and normalized so that
$\phi(e) =1$, then this function must
be of the form $\phi_v$ for some unitary representation $(\cH_\phi,\pi_\phi)$ of 
$G$ with a $\pi(H)$-invariant (unit) vector $v$. To see this, we first note that we can
equip the linear space $C_o(G/H)$ of all complex-valued finitely  supported 
functions on the coset space $G/H$ with the 
positive definite inner product 
\[
\langle c, c' \rangle_\phi = \sum_{x,y \in G/H} c(x) \overline{c'(y)} \, \phi(x^{-1}y), \quad c, c' \in C_o(G/H).
\]
One readily checks that the left regular representation of $G$ on 
$C_o(G/H)$ is unitary with respect to the inner product $\langle \cdot, \cdot \rangle_\phi$.
If we quotient out the (invariant) linear subspace of all elements $c$ which
satisfy $\langle c,c \rangle_\phi = 0$, and complete, then we have constructed a Hilbert space $\cH_\phi$ and a unitary representation 
$\pi_\phi$ on $\cH_\phi$ with a $\pi_\phi(H)$-invariant 
(cyclic) unit vector $v$ (which correspond to the equivalence class of the indicator function of the identity coset in $G/H$) such that 
\[
\phi(HgH) 
= 
\langle v, \pi(g)v \rangle_\phi, 
\quad 
\textrm{for all $HgH \in H \backslash G / H$}. 
\] 
In the literature, the unitary representation $(\cH_\phi,\pi_\phi)$ is often referred to as the \emph{GNS-construction} of the positive definite 
function $\phi$. \\

We note that the GNS-construction associated to the 
positive definite function $\phi$ on $H \backslash G / H$ which satisfies $\phi(HeH) = 1$
and $\phi(HgH) = 0$ for all $g \neq HeH$ is isomorphic 
to the left regular representation of $G$ on $\ell^2(G/H)$,
and the GNS-construction associated to the constant function one
is simply the identity representation. These unitary representations
of course exist for \emph{any} countable group $G$ and subgroup
$H$ thereof. If we want to construct other unitary representations, we should
impose more conditions on the pair $(G,H)$.

\subsection{Higher order characters}
Let $G_o$ be a countable group and define for $k \geq 2$, the pair
\[
G = G_o^k \qand H = \Delta_k(G_o),
\]
where the latter group denotes the diagonal subgroup in $G_o^k$. We
note that the coset space $G/H$ can be identified with $G_o^{k-1}$, and 
if $k = 2$, then a bi-$H$-invariant positive definite function on $G$ is
nothing else than a conjugation-invariant positive definite function on 
$G_o$ (often called a character). 

The study of characters on countable (non-abelian) discrete groups is 
quite subtle in general and complete classifications are only known in a
hand-full of cases. However, this research area has recently 
experienced a rejuvenation with a series of breakthrough results and 
observations; see for instance the papers \cite{DM14}, \cite{PT13} and 
\cite{V11}. \\

The cases when $k \geq 3$ seem to have attracted much 
less attention. 

\subsection{Semidirect products}
Let us now consider the case when 
\[
G = N \rtimes \Lambda \qand H = \{o\} \rtimes \Lambda,
\]
where $N$ is a countable \emph{abelian} group and 
$\Lambda < \Aut(N)$. The multiplication on $G$ can be 
encoded on the direct product $N \times \Lambda$ by
\[
(n_1,\lambda_1)(n_2,\lambda_2) = (n_1 + \lambda_1(n_2), \lambda_1 \lambda_2)
\quad
\textrm{for all $(n_1,\lambda_1), (n_2,\lambda_2) \in N \times \Lambda$.}
\]
If $\phi$ is a bi-$H$-invariant positive definite function on $G$, then
\[
\phi(n,\lambda) = \phi_o(n) \quad \textrm{for all $(n,\lambda) \in G$}
\]
for some positive definite function $\phi_o$ on $N$ which is invariant 
under the action of $\Lambda$. Since $N$ is abelian, a classical theorem 
of S. Bochner (see e.g.  the book \cite{Ru}) now asserts that there exists a
\emph{unique} probability measure $\nu_o$ on the dual (compact) group $\hat{N}$ of $N$, which is invariant under the dual action of $\Lambda$ on $\hat{N}$, such
that
\[
\phi_o(n) = \int_{\hat{N}} \chi(n) \, d\nu_o(\chi), \quad \forall \, n \in N.
\] 
Conversely, if $\nu_o$ is a probability measure on $\hat{N}$ which is invariant under the dual action of $\Lambda$, then the function 
$\phi(n,\lambda) = \phi_o(n)$ is a bi-$H$-invariant positive definite on 
$G$ and thus corresponds via the GNS-construction to a unitary 
representation of $G$ with a (cyclic) $H$-invariant vector. \\

The study of $\Lambda$-invariant probability measures on $\bT^d$ 
for various (thin) subgroups $\Lambda$ of the automorphism group 
$\Aut(\bZ^d) = \GL_d(\bZ)$ has attracted a lot of interest in recent 
years and very little is known in general.

\subsection{Hecke pairs}

Let $G$ be a countable group. We say that a subgroup $H < G$ is 
\emph{almost normal} (or commensurated) if every $H$-orbit
on $G/H$ is finite. In particular, every normal subgroup of $G$
is almost normal. If a subgroup $H$ is almost normal, then one
often refers to $(G,H)$ as a \emph{Hecke pair}. \\

Given a Hecke pair $(G,H)$ one can always associate a pair 
$(\overline{G},\overline{H})$ (known as the \emph{Schlicting 
completion} of $(G,H)$), where $\overline{G}$ is a locally compact
group and $\overline{H}$ is a compact open subgroup thereof, in such
a way that $G$ embeds densely in $\overline{G}$ and there is a 
one-to-one correspondence between the bi-$H$-invariant positive definite functions on $G$ and bi-$\overline{H}$-invariant continuous positive 
definite functions on $\overline{G}$. For instance, if
\[
G = \PSL_2(Z[1/p])
\qand 
H = \PSL_2(\bZ),
\]
where $p$ is a prime number, then
\[
\overline{G} = \PSL_2(\bQ_p)
\qand 
\overline{H} = \PSL_2(\bZ_p).
\]
We refer the reader to the paper \cite{A14} for a concise treatment of Hecke
pairs and their Schlicting completions. In the example above, the unitary representations of $\PSL_2(\bQ_p)$ with a non-zero 
$\PSL_2(\bZ_p)$-invariant vector (also known as \emph{class one} representations) are well-known and admit explicit descriptions; see e.g. 
the book \cite{Lu}.

\section{Proof of the Left Weak Ergodic Theorem}
\label{sec:lwet}
The proof of Theorem \ref{lwet} is not hard and we include a proof only 
for completeness. Let $(G,H,L)$ be a left amenable triple and suppose 
that $(F_n)$ is a a left F\o lner sequence for this triple. We wish to show
that if $(\cH,\pi)$ is a (separable) unitary representation with a $\pi(H)$-invariant unit vector $v$, then 
\[
\lim_n \frac{1}{|F_n|} \sum_{s \in F_n} \pi(s)v = P_L v
\]
in the weak topology on $\cH$. Since the unit ball in $\cH$ is weakly \emph{sequentially} compact, it suffices to prove that whenever $(n_j)$ 
is a sequence such that the limit of the sequence
\[
v_j = \frac{1}{|F_{n_j}|} \sum_{s \in F_{n_j}} \pi(s)v
\]
exists in the weak topology, then this weak limit is $\pi(L)$-invariant, or 
equivalently, the difference $\pi(t)v_{j} - v_{j}$  converges to zero
in the weak topology for every $t \in L$. \\

We readily check that
\[
\big\| \pi(t)v_{j} - v_{j} \big\|
= 
\Big\| 
\frac{1}{|F_{n_j}|} 
\sum_{s \in F_{n_j} \triangle tF_{n_j}} \pi(s) v
\Big\|
\leq \frac{|F_{n_j} \triangle tF_{n_j}|}{|F_{n_j}|},
\]
and the right hand side tends to zero with $j$ by the left F\o lner property
for $(F_n)$, which finishes the proof. 

\section{Proof of the Right Strong Ergodic Theorem}
\label{sec:rset}

The aim of this section is to give a proof of Theorem \ref{rset} based on
an adaption of the elegant argument of F. Riesz outlined in \cite{R45}. \\

The main part of F. Riesz argument is an important observation about 
decompositions of unitary representations which in the case of triples 
we have adapted as follows.

\begin{proposition}[Orthogonal splitting]
\label{osplit}
Let $G$ be a countable group and suppose that $H$ and $L$ are 
subgroups of $G$ with $H < L < N_G(H)$. If $(\cH,\pi)$ is a 
unitary representation of $G$, then
\begin{equation}
\label{prop. orth. split}
\cH^H = \cH^L \oplus \cobnd_{(L,H)}(\cH,\pi)
\end{equation}
where
\[
\cobnd_{(L,H)}(\cH,\pi) = \overline{\linspan\big\{ v - \pi(s)v \, : \, s \in L, \: v \in \cH^H \big\}} \subset \cH^H.
\]
\end{proposition}

Once this result has been established, the proof of Theorem \ref{rset} is 
immediate. We recall that we wish to prove that if $(F_n)$ is a right 
F\o lner sequence for a right amenable triple $(G,H,L)$ and $(\cH,\pi)$
is a unitary representation of $G$ with a $\pi(H)$-invariant unit vector
$v$, then
\[
\lim_n \frac{1}{|F_n|} \sum_{s \in F_n} \pi(s)v = P_L v
\]
in the norm topology on $\cH$. We note that it is enough to prove this
for a set of vectors $v$ in $\cH$ which span a norm-dense subspace of 
$\cH$. The proposition above asserts that the set 
\[
S = \big\{ v + w - \pi(s_o)w \, : \, v \in \cH^{L}, w \in \cH^{H} \qand s_o \in L \big\}
\]
spans a norm-dense subspace of $\cH$, and 
\[
\Big\|
\frac{1}{|F_n|} \sum_{s \in F_n} \big(\pi(s)v + \pi(s)w - \pi(ss_o)w\big) - v
\Big\|
\leq 
\frac{|F_n \triangle \alpha_H(s_o)F_n|}{|F_n|} \ra 0,
\]
as $n$ tends to infinity, since $\pi(s)v = v$ for all $s \in F_n \subset L/H$
and $P_L(v+w-\pi(s_o)w) = v$ and $(F_n)$ is a 
right F\o lner sequence for $(G,H,L)$, 
which finishes the proof of Theorem 
\ref{rset}.

\begin{proof}[Proof of Proposition \ref{osplit}]
We note that it suffices to show that if $v \in \cH^H$ is orthogonal to all
elements of the form $w-\pi(s)w$ with $w \in \cH^H$ and $s \in L$, then
$w$ is $\pi(L)$-invariant. Since
\[
\langle v, w-\pi(s)w \rangle = \langle v-\pi(s^{-1})v,w \rangle = 0
\quad \textrm{for all $w \in \cH^H$ and $s \in L$},
\]
and 
\[
v - \pi(s^{-1})v \in \cH^H \quad \textrm{for all $s \in L < N_G(H)$},
\]
we conclude
that $v = \pi(s)v$ for all $s \in L$.
\end{proof}

\section{Proof of the Transfer Principle}
\label{sec:transfer}
We note that if the Left Strong Ergodic Theorem holds for a countable 
amenable group $G$, then every left F\o lner sequence in $G$ is a
firm strong sequence, so to prove Proposition \ref{transfer0}
it suffices to establish the following "local" transfer principle: 

\begin{proposition}[Local transfer principle]
Let $G$ be a countable amenable group and let $(\cH,\pi)$
be a unitary representation of $G$. Suppose that there exists
at least one firm strong sequence (not necessarily a left F\o lner
sequence) in $G$  with respect to $(\cH,\pi)$. 
Then the Left Strong Ergodic Theorem holds for $(\cH,\pi)$.
\end{proposition}

We shall begin by reducing the proof of the proposition to a general 
lemma about left F\o lner sequences in amenable groups which 
should be interesting in its own right. \\

Suppose that we are given a unit vector $v \in \cH$ and define the bounded real-valued function
\[
\phi(s,t) = \Re \langle \pi(s)v, \pi(t)v \rangle, 
\quad 
\textrm{for $(s,t) \in G \times G$},
\]
where $\Re$ denotes the real part of a complex number. \\

Let us now suppose that $(K^o_m)$ is a firm strong sequence in $G$ with respect to the unitary representation $(\cH,\pi)$ of $G$. We note that for 
every pair of sequences $(s^o_m)$ and $(t^o_m)$ in $G$, 
\begin{eqnarray*}
 \Big| \frac{1}{|K^o_m|^2} 
\sum_{(s,t) \in K^o_m \times K^o_m}
\phi(ss^o_m,tt^o_m) \Big|
&=&
\Big| \Re 
\Big\langle  
\frac{1}{|K^o_m|} \sum_{s \in K^o_m} \pi(ss^o_m)v, 
\frac{1}{|K^o_m|} \sum_{t \in K^o_m} \pi(tt^o_m)v
\Big\rangle \Big| \\
&\leq &
\Big\| \, \frac{1}{|K^o_m|} \sum_{s \in K^o_m} \pi(ss^o_m)v \, \Big\|
\ra 0.
\end{eqnarray*}
We wish to prove that if 
$(F_n^o)$ is a left F\o lner sequence in $G$, then 
\[
\varlimsup_n 
\Big\| \frac{1}{|F_n^o|} \sum_{s \in F^o_n} \pi(s)v\Big\|^2 
= 
\varlimsup_n 
\frac{1}{|F_n^o|^2}\sum_{(s,t) \in F^o_n \times F^o_n} \phi(s,t) = 0,
\]
where the first equality sign stems from 
\[
\sum_{(s,t) \in F^o_n \times F^o_n} 
\langle \pi(s) v, \pi(t) v \rangle
= 
\sum_{(s,t) \in F^o_n \times F^o_n} 
\frac{1}{2} \big( \langle \pi(s) v, \pi(t) v \rangle + \langle \pi(t) v, \pi(s) v \rangle\big)
= 
\sum_{(s,t) \in F_n^o \times F_n^o} \phi(s,t),
\]
since the inner product on $\cH$ is skew-hermitian. \\

Proposition \ref{transfer0} is now an immediate consequence of  
the following lemma (which is a variation of Lemma 3.3 in \cite{BBF}) 
applied to the (amenable) direct product group $G \times G$, the sequences 
\[
K_m = K_m^o \times K_m^o \qand F_n = F_n^o \times F_n^o,
\]
in $G \times G$ and the bounded real-valued function $\phi$ on $G \times G$ defined above. We note that $(F_n)$ is a left F\o lner sequence in $G \times G$.

\begin{lemma}
\label{transferlemma}
Let $G$ be a countable amenable group and 
fix a left F\o lner sequence $(F_n)$ in $G$ and a bounded 
real-valued function $\phi$ on $G$.  Then, for every sequence $(K_m)$ 
of finite subsets of $G$, there exists a sequence $(t_m)$ in $G$ such 
that 
\[
\varlimsup_n \frac{1}{|F_n|} \sum_{t \in F_n} \phi(t)
\leq
\varliminf_m \frac{1}{|K_m|} \sum_{s \in K_m} \phi(st_m).
\]
\end{lemma}

\begin{proof}
We extract a subsequence $(n_j)$ such that
\[
\beta = \varlimsup_n \frac{1}{|F_n|} \sum_{t \in F_n} \phi(t) = 
\lim_j \frac{1}{|F_{n_j}|} \sum_{t \in F_{n_j}} \phi(t).
\]
Since $(F_{n_j})$ is a \emph{left} F\o lner sequence in $G$ and $\phi$ is bounded and real-valued, we can find, for every integer $m$, an index $j$ such that
\[
\frac{1}{|F_{n_j}|} \sum_{t \in F_{n_j}} \phi(st) \geq \beta - \frac{1}{m}
\]
for all $s \in K_m$, and thus
\[
\frac{1}{|K_m|} \sum_{s \in K_m} 
\Big( 
\frac{1}{|F_{n_j}|} \sum_{t \in F_{n_j}} \phi(st) 
\Big)
= 
\frac{1}{|F_{n_j}|} \sum_{t \in F_{n_j}} 
\Big( 
\frac{1}{|K_m|} \sum_{s \in K_m} 
\phi(st) 
\Big)
\geq 
\beta - \frac{1}{m}.
\]
In particular, there must exist at least one element $t_m \in F_{n_j}$ such that
\[
\frac{1}{|K_m|} \sum_{s \in K_m} 
\phi(st_m)  \geq \beta - \frac{1}{m},
\]
which shows that 
\[
\varliminf_m \frac{1}{|K_m|} \sum_{s \in K_m} 
\phi(st_m)  \geq \beta.
\]
\end{proof}

\section{A Left Strong Ergodic Theorem for nilpotent groups}
\label{sec:lseticc}

The aim of this section is to prove Theorem \ref{lset}, but before we can 
do this we need to show that every countable group $G$ without ICC quotients can be decomposed as a (possibly) transfinite increasing 
chain of \emph{normal} subgroups $(G_\alpha)$ so that for every 
unitary representation $(\cH,\pi)$ of the group $G$, the Left Strong Ergodic Theorem 
is immediate for elements in the unitary \emph{sub-representation} 
$\cobnd_{(G,G_\alpha)}(\cH,\pi)$ for every $\alpha$. By a 
rather straightforward approximation argument, this yields the Left 
Strong Ergodic Theorem for the whole group $G$. \\

Let $G$ be a countable group and let $H$ be a \emph{normal} 
subgroup of $G$. We denote by $p_H$ the canonical quotient 
homomorphism from $G$ onto $G/H$, and for $x \in G$ we 
write $x^G$ for the conjugation class of $x$, i.e. the set of all
conjugates of $x$ in $G$. We define 
\[
\FC_G(H) = \big\{ x \in G \, : \, p_H(x^G) \subset G/H \quad  \textrm{is finite} \big\},
\]
and we note that since $H$ is normal, $\FC_G(H)$ is again a 
\emph{normal} subgroup of $G$. In particular, if $H$ is the 
trivial subgroup, then $\FC_G(H)$ consists of exactly those 
elements in $G$ with a finite conjugation class. \\

We can also 
iterate this construction: Set $\FC^{1}_G(H) = \FC_G(H)$ and
for every integer $k \geq 1$, we define
\[
\FC_G^{(k+1)}(H) = \FC_G(\FC_G^{(k)}(H)).
\]
More generally, if $\alpha$ is an ordinal, then we define
\[
\FC_G^{(\alpha+1)}(H) = \FC_G(\FC_G^{(\alpha)}(H))
\qand 
\FC^{(\alpha)}(G) = \FC^{(\alpha)}_G(\{e\}).
\]
We note that $\FC^{(\alpha)}_G(H) \subset \FC_{G}^{(\alpha+1)}(H)$ for every 
ordinal $\alpha$, so if $\lambda$ is a limit ordinal, then it is 
well-defined to set 
\[
\FC_G^{(\lambda)}(H) = \bigcup_{\alpha < \lambda} \FC_G^{(\alpha)}(H).
\]
Since $G$ is a countable, the least ordinal $\alpha_H$ such that
\[
\FC_G^{(\alpha_H+1)}(H) = \FC_G^{(\alpha_H)}(H)
\]
is countable, and we set $G_{\textrm{icc}} = \FC^{\alpha_{\{e\}}}(G)$. We 
note that every non-identity element (if any) in the  quotient group 
$G/G_{\textrm{icc}}$ must have an infinite conjugation class, so we have
thus proved the following proposition.

\begin{proposition}[Characterization of ICC quotients]
A countable group $G$ does not admit any ICC quotients if and only if there exists an ordinal $\alpha$ such that $\FC^{\alpha_{\{e\}}}(G) = G$.
\end{proposition}

\begin{remark}
We note that every countable group all of whose conjugation classes are 
finite is amenable. Indeed, if $L$ is such a group, then there exists an 
increasing exhaustion $(F_n)$ of $L$ by \emph{finite} conjugation-invariant
subsets. Hence we can identify the image of the conjugation action map 
$p : L \ra \Sym(L)$ with the increasing union of the \emph{finite} groups 
$\Sym(F_n)$, which is clearly amenable. Since the kernel of the action map 
$p$ equals the center of $L$ (which is abelian), we conclude that $L$ is 
amenable.  

Now, if $G$ is any countable group and $H$ is a normal subgroup of $G$, then $\FC_G(H)/H$ is clearly isomorphic to $\FC_{G/H}(\{e\})$, which is amenable 
by the discussion above, and thus $\FC_G(H)$ is amenable whenever $H$ is amenable. 
\end{remark}

\begin{example}[Nilpotent groups]
We recall that a countable group is \emph{nilpotent} if its upper central
series terminates after a finite number of steps, that is to say, if we 
define $Z_o = \{e\}$ and
\[
Z_{i+1} 
= 
\big\{ 
s \in G \, : \, [s,t] \in Z_{i} \quad \textrm{for all $t \in G$}
\big\},
\]
then there exists an integer $k$ such that
\[
\{e\} = Z_o \triangleleft Z_1 \triangleleft \ldots \triangleleft Z_k = G.
\]
Since $Z_i \subset \FC^{(i)}(G)$ for every $i$, we have $G = \FC^{(k)}(G)$ as well. More generally, we say that $G$ is \emph{virtually nilpotent} if it contains a nilpotent subgroup $G_o$ of finite index. We leave it as an 
exercise to verify that $\FC^{(k)}(G) = G$ whenever $\FC^{(k)}(G_o) = G_o$,
and thus virtually nilpotent groups do not admit any ICC quotients either.
\end{example}

We shall now turn to the proof of Theorem \ref{lset}. Let $G$ be a countable
group which does not admit any ICC quotient groups and suppose that
$(\cH,\pi)$ is a unitary representation of $G$. By Proposition \ref{osplit} 
we can write 
\[
\cH 
= 
\cH^{G_1} \oplus \cobnd_{(G_1,G_o)}(\cH,\pi)
=
\ldots
= 
\cH^{G_{k-1}} 
\oplus 
\Big( 
\bigoplus_{i=0}^{k-2} \cobnd_{(G_{i+1},G_i)}(\cH,\pi) 
\Big)
\]
for every $k$, where $G_i = \FC^{(i)}(G)$. For a fixed integer $k$, we set 
\[
\cH_k = \bigoplus_{i=0}^{k-2} \cobnd_{(G_{i+1},G_i)}(\cH,\pi) 
\]
and define inductively for any ordinal $\alpha$,
\[
\cH_{\alpha+1} = \cH_{\alpha} 
\oplus 
\cobnd_{(G_{\alpha+1},G_{\alpha})}(\cH,\pi). 
\]
If $\lambda$ is a limit ordinal, we set
\[
\cH_\lambda = \overline{\bigcup_{\alpha < \lambda} \cH_\alpha} \subset \cH.
\]
If we assume that the representation $(\cH,\pi)$ does not have any non-zero $\pi(G)$-invariant 
vectors, then $\cH = \cH_\lambda$, where $\lambda$ is the 
smallest ordinal such that $G = G_{\lambda}$. \\

We wish to prove that if $(F_n)$ is any \emph{right} F\o lner sequence in $G$
and $(s_n)$ is any sequence in $G$, then 
\begin{equation}
\label{check}
\varlimsup_n \Big\|\frac{1}{|F_n|} \sum_{s \in F_n} \pi(ss_n)v \Big\| = 0,
\quad \textrm{for all $v \in \cH$}
\end{equation}
We note that it is enough to establish these limits for a
dense subspace of $\cH$, so in particular we only need to calculate the limit
above for vectors of the form $v_\alpha - \pi(t_\alpha)v_\alpha$, 
where  $v_\alpha$ belongs to $\cH^{G_\alpha}$ and 
$t_\alpha$ is in $G_{\alpha+1}$ for some ordinal $\alpha$. \\

We fix a vector of this form and assume that \eqref{check} fails for some sequence $(s_n)$, that is to say, possibly upon passing to a subsequence, 
we assume that
\[
\lim_n \Big\|\frac{1}{|F_n|} 
\sum_{s \in F_n} 
\big(\pi(ss_n)v_\alpha - \pi(ss_n)\pi(t_\alpha)v_\alpha\big) \Big\| > 0.
\]
We recall that since $t_\alpha \in G_{\alpha+1}$, the set
\[
S_\alpha = \big\{ p_{G_\alpha}\big(s_nt_\alpha s_n^{-1}\big) \, : \, n \geq 1 \big\}
\subset G/G_\alpha
\]
is finite, and thus, possibly upon passing to a further subsequence, we may 
assume that the identities $s_n t_\alpha s_n^{-1} = u_\alpha$ hold modulo $G_\alpha$ for 
all $n$ and for some element $u_\alpha \in G_{\alpha+1}$. Hence, along this
subsequence, we have
\[
\frac{1}{|F_n|} 
\sum_{s \in F_n} 
\big(\pi(ss_n)v_\alpha - \pi(ss_n)\pi(t_\alpha)v_\alpha\big)
= 
\frac{1}{|F_n|} 
\sum_{s \in F_n} 
\big(\pi(s) - \pi(su_\alpha)\big) \, \pi(s_n) v_\alpha,
\]
since $v_\alpha$ is $G_\alpha$-invariant, and we conclude that
\[
\Big\|
\frac{1}{|F_n|} 
\sum_{s \in F_n} 
\big(\pi(s) - \pi(su_\alpha)\big) \, \pi(s_n) v_\alpha
\Big\|
= 
\Big\|
\frac{1}{|F_n|} 
\sum_{s \in F_n \triangle F_n u_\alpha} \pi(ss_n)v_\alpha 
\Big\|
\leq \frac{|F_n \triangle F_n u_\alpha|}{|F_n|} \ra 0,
\]
since $(F_n)$ is a right F\o lner sequence in $G$. This contradiction 
finishes the proof of Theorem \ref{lset}. \\

In fact, the arguments above also yield the following strengthening of 
Theorem \ref{lset}. 

\begin{proposition}
\label{orticc}
Let $G$ be a countable amenable group and suppose that $(\cH,\pi)$
is a unitary representation of $G$. Then there exists a normal 
subgroup $G_{icc}$ of $G$ with the property that the quotient group
$G/G_{icc}$ is either 
trivial or ICC and a closed $\pi(G)$-invariant 
subspace $\cH_o \subset \cH$ such that 
\[
\cH = \cH^{G_{icc}} \oplus \cH_o,
\]
and the Left Strong Ergodic Theorem holds for $(\cH_o,\pi)$. 
\end{proposition}
\subsection{An outline of the proof of Corollary \ref{redicc}}
\label{subsec:redicc}
Let $G$ be a countable
amenable group and assume that $(\cH,\pi)$ is a unitary representation of 
$G$ with no non-zero $\pi(G)$-invariant vectors. By Proposition 
\ref{orticc} there exists a closed $\pi(G)$-invariant subspace $\cH_o$ of
$\cH$ so 
that the Left Strong Ergodic Theorem holds for $(\cH_o,\pi)$ and 
\[
\cH = \cH^{G_\textrm{icc}} \oplus \cH_o,
\]
where $G_{icc} < G$ is a normal subgroup with the property that the quotient 
group $G/G_{icc}$ is either trivial or ICC. \\

In particular, if we assume that there exists a left F\o lner sequence $(F_n)$
in $G$ such that the Left Strong Ergodic Theorem fails for 
$(\cH,\pi)$ along this sequence, then we can conclude that 
$G_{\textrm{icc}} \neq G$ and that the Left Strong Ergodic Theorem fails 
for $(\cH^{G_{icc}},\pi)$ along the sequence $(F_n)$. That is to say, for some unit vector $v \in \cH^{G_{icc}}$, we have
\[
\delta = \varlimsup_n \Big\| \frac{1}{|F_n|} \sum_{s \in F_n} \pi(s)v \Big\| > 0.
\]
Clearly, the unitary representation $(\cH^{G_{icc}},\pi)$ of $G$ factors through 
a unitary representation $\pi'$ on $\cH' = \cH^{G_{icc}}$ of the quotient 
group $G/G_{icc}$, and we can write
\[
\frac{1}{|F_n|} \sum_{s \in F_n} \pi(s)v = \sum_{t \in G/G_{icc}} \beta_n(t) \, \pi'(t)v,
\]
where
\[
\beta_n(tG_{icc}) = \frac{1}{|F_n|} \, \sum_{s \in G_{icc}} \chi_{F_n}(ts),
\quad 
\textrm{for all $tG_{icc} \in G/G_{icc}$}.
\]
One can readily verify that $(\beta_n)$ satisfies
\[
\lim_n \sum_{t \in G/G_{icc}} \big| \beta_n(ut) - \beta_n(t) \big| = 0
\quad
\textrm{for all $u \in G/G_{icc}$},
\]
and 
\[
\varlimsup_n \sum_{s,t \in G/G_{icc}} \beta_n(s) \, \beta_n(t) \, \, 
\Re \langle \pi'(s) v, \pi'(t) v \rangle = \delta^2 > 0,
\]
where $\Re$ denotes the real part of a complex number. \\

In particular, any accumulation point $\lambda$ of the sequence 
\[
\beta_n \otimes \beta_n  
\in \ell^1(G/G_{icc} \times G/G_{icc}) \subset \ell^1(G/G_{icc} \times G/G_{icc})^{**}
\] 
is invariant under left translations by elements in 
the direct product group $G/G_{icc} \times G/G_{icc}$. These limits 
are often referred to as \emph{left invariant means} on 
$G/G_{icc} \times G/G_{icc}$, and we note that 
\[
\lambda(\phi) \geq \frac{1}{2} \cdot \delta^2,
\] 
where $\phi$ is the bounded real-valued function on $G/G_{icc} \times G/G_{icc}$ defined by
\[
\phi(s,t) = \Re \, \langle \pi'(s) v, \pi'(t) v \rangle, \quad 
\textrm{for all $(s,t) \in G/G_{icc} \times G/G_{icc}$}.
\]
We stress that this expression is well-defined since $v$ is assumed to 
be invariant under $G_{icc}$. \\

A classical argument by I. Namioka in \cite{Nam} will now guarantee that
we can find a left F\o lner sequence $(L_n)$ in $G/G_{icc} \times G/G_{icc}$ such that 
\[
\varlimsup_n \frac{1}{|L_n|} \sum_{(s,t) \in L_n} \phi(s,t) \geq \frac{1}{2} \cdot \delta^2.
\]
Let $(F'_n)$ be any left F\o lner sequence in $G/G_{icc}$. By Lemma \ref{transferlemma} we
can now find a sequence of elements $(s_n,t_n)$ in the direct product 
group $G/G_{icc} \times G/G_{icc}$ such
that
\[
0 < \frac{1}{2} \cdot \delta^2 \leq \varlimsup_n \frac{1}{|L_n|} \sum_{(s,t) \in L_n} \phi(s,t) 
\leq 
\varliminf_n \frac{1}{|F'_n|^2} \sum_{(s,t) \in F'_n \times F'_n} \phi(ss_n,tt_n). 
\]
Assume that the Left Strong Ergodic Theorem holds for the unitary 
representation $(\cH^{G_{icc}},\pi')$ of $G/G_{icc}$. Then, 
\[
\lim_n \frac{1}{|F'_n|} \sum_{s \in F'_n} \pi(ss_n)v = 0
\qand
\lim_n \frac{1}{|F'_n|} \sum_{s \in F'_n} \pi(st_n)v = 0
\]
in the norm topology on $\cH$, which implies that 
\[
\frac{1}{|F'_n|^2} \sum_{(s,t) \in F'_n \times F'_n} \phi(ss_n,tt_n).
=
\Big\langle 
\frac{1}{|F'_n|} \sum_{s \in F'_n} \pi'(ss_n)v, \frac{1}{|F'_n|} \sum_{s \in F'_n}\pi'(st_n)v
\Big\rangle
= 0.
\]
This contradiction shows that the Left Strong Ergodic Theorem for $(\cH^{G_{icc}},\pi')$ must fail
at least along one of the left F\o lner sequences $(F'_n s_n)$ or $(F'_n t_n)$,
and finishes the proof of Corollary \ref{redicc}.

\section{A Left Strong  Ergodic Theorem for induced representations}
\label{sec:induced}
The aim of this section is to prove Proposition \ref{induced}. We begin
by recalling some basic properties of induced representations. \\

Let $G$ be a countable group. If $(\cH_o,\pi_o)$ is a unitary 
representation of a subgroup $L < G$, one can induce
a unitary representation of $G$ from $(\cH_o,\pi_o)$ as 
follows. We define the linear space
\[
\cH' = 
\Big\{ 
f : G \ra  \cH_o \, : \, \|f\| < \infty \qand f(lg) = \pi_o(l)^{-1}f(g), 
\quad 
\textrm{for all $l \in L$ and $g \in G$}
\Big\}.
\]
Here, $\|\cdot\|$ denotes the norm induced from the following 
inner product 
\[
\langle f,f' \rangle = \sum_{x \in L \backslash G} \langle f(x), f'(x) \rangle_o,
\]
on $\cH$, where $\langle \cdot , \cdot \rangle_o$ is the inner product on 
$\cH_o$. We readily check that $\cH'$ is a pre-Hilbert space and we denote
by $\cH$ the Hilbert space completion on $\cH'$. We note that $G$ 
acts on $\cH'$ (and thus on $\cH$) via the representation
\[
(\pi(g)f)(x) = f(xg), \quad \textrm{for $g, x \in G$},
\]
which is unitary with respect to $\langle \cdot, \cdot \rangle$. This 
representation is often referred to as the \emph{induced representation}
of $(\cH_o,\pi_o)$ and is commonly denoted by $\Ind_L^G \pi_o$. \\

Suppose that we are given a sequence $(F_n)$ of finite subsets of $G$. We wish to prove that if $L$ is \emph{thin with respect to $(F_n)$}, by which 
we mean that
\begin{equation}
\label{thinsubgroup2}
\lim_n \sup_{x, y \in G/L} \frac{|F_n \cap xLy^{-1}|}{|F_n|} = 0,
\end{equation}
then $(F_n)$ is a firm strong sequence with respect to the induced representation $(\cH,\pi)$, that is to say, for every sequence $(s_n)$ 
in $G$, we have
\begin{equation}
\label{firm2}
\lim_n \frac{1}{|F_n|} \sum_{s \in F_n} \pi(ss_n)v = 0,
\end{equation}
in the norm topology on $G$. The assumption that the right hand side  above should vanish is justified once we notice that \eqref{thinsubgroup2}  forces the index of $L$ in $G$ to be infinite and that induced 
representations from an infinite index subgroup $L$ never have non-zero
$\pi(G)$-invariant vectors. \\

The following lemma is quite standard and we give the proof for 
completeness.

\begin{lemma}
\label{decomp}
For every sequence of orthogonal vectors $(v_k)$ in $\cH_o$ such that 
\[
\cH_o = \bigoplus_{k} \, \overline{\linspan\big\{ \pi_o(l)v_k \, : \, l \in L \big\}}
\]
we have
\[
\cH = \bigoplus_{k} \, \overline{\linspan\big\{ \pi(g)f_k \, : \, g \in G \big\}},
\]
where the function $f_k : G \ra \cH_o$ is defined by
\[
f_k(x) = 
\left\{
\begin{array}{cc}
\pi_o(x)^{-1}v_k & \textrm{if $x \in L$} \\
0 &  \textrm{if $x \notin L$}
\end{array}
\right.
.
\]
In particular, if $(\cH_o,\pi_o)$ is cyclic, then so is $(\cH,\pi)$.
\end{lemma}

\begin{proof}
We note that once we have checked that $f_k$ is indeed an element in the
induced representation $\cH$ for every $k$, then it suffices to show that if 
$f \in \cH$ satisfies
\[
\langle f, \pi(s)f_k \rangle = 0 \quad \textrm{for all $s \in G$},
\]
and for every index $k$, then $f(s)$ is zero in $\cH_o$ for every $s \in G$. Upon expanding the inner product, we see that this condition exactly means that
\[
\langle f(s), v_k \rangle_o = 0 \quad \textrm{for all $s \in G$ and for every $k$},
\]
since $f_k$ vanishes off $L$. In particular, for every $s \in G$, we have
\[
\langle f(ls), v_k \rangle_o = \langle f(s), \pi_o(l) v_k \rangle_o = 0 \quad \textrm{for all $l \in L$ and for every $k$},
\]
which forces $f(s) = 0$ for all $s \in G$, since every $v_k$ is cyclic for a 
sub-representation of $\cH_o$ and together all of these sub-representation
are assumed to span $\cH_o$.
\end{proof}

We now fix an orthogonal decomposition of $\cH_o$ as in Lemma \ref{decomp} with \emph{unit} 
vectors $(v_k)$. We readily verify that the 
positive definite functions associated to the elements $(f_k)$ in $\cH$
have the form:
\[
\phi_k(g) = \langle f_k, \pi(g) f_k \rangle
= 
\left\{
\begin{array}{cc}
\langle v_k, \pi_o(g) v_k \rangle_o & \textrm{if $g \in L$} \\
0 &  \textrm{if $g \notin L$}
\end{array}
\right..
\]
It is not hard to see that if suffices to establish \eqref{firm2} for $f_k$ for every $k$, 
which explicitly means that one has to prove that for every sequence 
$(s_n)$ in $G$,  
\[
\varlimsup_n 
\Big\|\frac{1}{|F_n|} \sum_{s \in F_n} \pi(ss_n)f_k \Big\|^2
= 
\varlimsup_n \frac{1}{|F_n|^2} \sum_{(s,t) \in F_n \times F_n} \phi_k(s_n^{-1}s^{-1}ts_n)
= 0.
\]
We note that in the case when $L$ is thin with respect to $(F_n)$, then this  
is a straightforward consequence of  the following lemma.

\begin{lemma}
For every finite subset $F \subset G$, we have
\[
\sup_{u \in G} 
\Big|
\frac{1}{|F|^2} 
\sum_{(s,t) \in F \times F} 
\phi_k(u^{-1}s^{-1}tu) 
\Big|
\leq 
\sup_{x,y \in G/L} \frac{|F \cap xLy^{-1}|}{|F|}
\]
for all $k$.
\end{lemma}

\begin{proof}
We first observe that if $F \subset G$ is a finite subset of $G$, then
\[
F^{-1}F \cap L = \bigcup_{x \in G/L} (F \cap xL)^{-1} (F \cap xL),
\]
and thus, for every $u$ in $G$, 
\[
\sum_{s \in F u} \sum_{t \in F u} 
\langle \pi(s)f_k,\pi(t)f_k \rangle
= 
\sum_{x \in G/L}
\sum_{s \in F u \cap xL} \sum_{t \in F  u \cap xL} 
\langle \pi(s) f_k,\pi(t)f_k \rangle,
\]
since $\langle \pi(s)f_k,\pi(t)f_k \rangle$ vanishes unless $s$ and $t$ belong
to the same  $F \cap xL$,  and these sets are disjoint as $x$ ranges over $G/L$.    
In particular, we have
\begin{eqnarray*}
\Big|
\frac{1}{|F|^2} \, \sum_{s \in F u} \sum_{t \in F  u} 
\langle \pi(s)f_k,\pi(t)f_k \rangle
\Big|
& = &
\Big|
\frac{1}{|F |^2} \, 
\sum_{x \in G/L}
\sum_{s \in F u \cap xL} \sum_{t \in F  u \cap xL} 
\langle \pi(s)f_k,\pi(t)f_k \rangle
\Big| \\
&\leq &
\sum_{x \in G/L} \Big( \frac{|F u \cap xL|}{|F  |}\Big)^2 \\
&\leq &
\sup_{x \in G/L} \frac{|F   \cap xLu^{-1}|}{|F  |} \cdot 
\sum_{x \in G/L} \frac{|F  \cap xLu^{-1}|}{|F |} \\
&\leq &
\sup_{x,y \in G/L} \frac{|F  \cap xLy^{-1}|}{|F |},
\end{eqnarray*}
where the first inequality sign holds since $\phi_k$ is bounded in absolute value by one. This finishes the proof.
\end{proof}

\subsection{Some remarks about semidirect products of groups}

We end this section with a few observations about subgroups of 
semidirect products of groups. \\

We say that a countable group $G$ \emph{splits as a semidirect product}
if there exist two subgroups $N$ and $A$ of $G$ with $N$ normal in $G$
such that $N \cap A$ is the trivial subgroup. This means that we can write 
every element $g$ in $G$ \emph{uniquely} as a product $g = na$ with 
$n \in N$ and $a \in A$, and upon identifying $G$ with $N \times A$ via 
the map $(n,a) \mapsto na$, we can encode the group multiplication in 
$G$ on $N \times A$ by
\[
(n_1,a_1)(n_2,a_2) = (n_1a_1n_2a_1^{-1},a_1a_2), \quad (n_1,a_1), (n_2,a_2) \in N \times A.
\]
We note that if $N$ and $A$ are amenable groups, then so is $G$, and if $(B_n)$ and $(A_n)$ are left F\o lner sequences in the subgroups $N$ and $A$ respectively, such that
\[
\limsup_n \frac{|aB_na^{-1} \triangle B_n|}{|B_n|} = 0, 
\quad \textrm{for all $a \in A$},
\]
then $(B_n A_n)$ is a left F\o lner sequence in $G$. \\

We shall prove below 
that under some natural conditions on $N$ and $A$, the Left Strong Ergodic Theorem holds for any unitary representation of $G$ which is induced from a unitary representation of a subgroup of either $N$ or $A$. \\

We first prove that unitary representations induced from \emph{normal} 
subgroups of a countable amenable group $G$ can never be counter-examples to the Left Strong Ergodic Theorem. 

\begin{proposition}
Let $G$ be a countable amenable group and let $N$ be a normal 
subgroup of $G$ of infinite index. Then $N$ is a thin subgroup 
with respect to every left F\o lner sequence in $G$, and thus every
such sequence is a firm strong sequence with respect to any unitary
representation induced from $N$.
\end{proposition}

\begin{proof}
Suppose that $N$ has infinite index in $G$ and that there exists a
left F\o lner sequence $(F_n)$ in $G$ such that $N$ is not thin with 
respect  to this sequence. We shall prove that this leads to a contradiction. Since $N$ is normal in $G$, we can find (possibly upon passing to a subsequence) a sequence $(x_n)$ in $G$ such that 
\[
\delta = \lim_n \frac{|F_nx_n \cap N|}{|F_n|} > 0.
\]
Since $(F_n)$ is a left F\o lner sequence, we can find, for every finite 
subset $K \subset G/N$, an index $n$ such that
\[
\frac{|F_nx_n \cap xN|}{|F_n|} \geq \frac{\delta}{2} \quad 
\textrm{for all $x \in K$},
\]
and thus
\[
1 
= 
\frac{|F_nx_n|}{|F_n|}
= 
\sum_{x \in G/N} \frac{|F_nx_n \cap xN|}{|F_n|} 
\geq
\sum_{x \in K} \frac{|F_nx_n \cap xN|}{|F_n|} \geq \frac{\delta}{2} \cdot |K|.
\]
Since $N$ has infinite index in $G$, we can choose $K \subset G/N$ so that 
the right hand side is strictly greater than one (for some index $n$), which is a contradiction.
\end{proof}

We shall see that the case when the unitary representation is induced from $A$ is slightly 
trickier, and the argument is quite different than the one outlined in the 
last proof. To simplify matters, we shall impose an unnecessarily strong 
assumption on the subgroup $A$. 

\begin{proposition}
Let $G$ be a countable  group which splits as a semi-direct 
product of two subgroups $N$ and $A$ of $G$
with $N$ normal in $G$. If the action of $A$ on 
$N \setminus \{e\}$ has trivial stabilizers, then $A$ is 
thin with respect to any sequence of the form $(B_n A_n)$,
where $(B_n)$ and $(A_n)$ are sequences of finite subsets
of $N$ and $A$ respectively whose sizes both tend to infinity.

In particular, any such sequence is a firm strong sequence for any unitary 
representation of $G$ which is induced from a unitary representation 
of $A$. 
\end{proposition}

\begin{remark}
The proposition above applies to the (countable) group 
$G_K$ of all affine transformations of a countable field $K$, which 
splits as a semidirect product of $K$ (translations) and $K^*$ (the multiplicative group of $K$ which acts by multiplication on $K$). 
The action of $K^*$ on $K \setminus \{0\}$ clearly has trivial stabilizers.

When $K = \bQ$ (or more generally
a finite extension thereof), J. Bost and A. Connes in their classical paper \cite{BC} 
constructed a family of \emph{irreducible} unitary representations of 
$G_K$ which are not induced from subgroups of either $K$ or $K^{*}$. 
We do not know if any of these representations constitutes a counter-example to the Left Strong Ergodic Theorem for $G_K$, and we refer to 
the more recent paper \cite{CBC} by T. Crisp for explicit formulas for 
these representations.  
\end{remark}

\begin{proof}
We note that since $G$ splits as a semidirect product of $N$ and $A$, it suffices to take the suprema in \eqref{thinsubgroup2} over $x, y \in N$. 
Thus we want to prove that if the action of $A$ on $N \setminus \{e\}$ has trivial stabilizers, then for every choice of sequences $(s_n)$ and $(t_n)$ in $N$, we have
\[
\varlimsup_n \frac{\big|B_n A_n \cap s_n A t_n\big|}{\big|B_n\big| \cdot \big|A_n\big|} \, = \,
\varlimsup_n \frac{\big|s_n^{-1}B_n A_n t_n^{-1} \cap A\big|}{\big|B_n\big| \cdot \big|A_n\big|}
=
0,
\]
whenever $(B_n)$ and $(A_n)$ are sequences of finite subsets of $N$ and $A$
respectively whose sizes tend to infinity. In other words, we wish to estimate the cardinality of the set of pairs $(b,a)$ in $B_n A_n$ which solves the equation
\[
s_n^{-1}bat_n^{-1} = (s_n^{-1}bat_n^{-1}a^{-1})a \in A.
\]
or equivalently, $at_na^{-1} = s_n^{-1}b$, since $A \cap N = \{e\}$. \\

We shall
separate between two cases. First, if $t_n = e$, then this set of solutions coincides with $s_nA_n$ and thus 
\[
\frac{\big|B_n A_n \cap s_n A t_n\big|}{\big|B_n\big| \cdot \big|A_n\big|}
= \frac{|s_n A_n|}{|A_n| \cdot |B_n|} = \frac{1}{|B_n|}.
\]
Since we assume that $|B_n|$ tends to infinity, we may without loss of generality assume that $t_n \neq e$ for all $n$. \\

In the second case, we assume that $t_n \neq e$ for all $n$. Since the 
action of $A$ on $N \setminus \{e\}$ has trivial stabilizers, there can be 
at most one solution to the equation $at_na^{-1} = s_n^{-1}b$ for every 
$s_n$ and $b$. Thus,
\[
\frac{\big|B_n A_n \cap s_n A t_n\big|}{\big|B_n\big| \cdot \big|A_n\big|}
\leq \frac{|s_n^{-1}B_n|}{|A_n| \cdot |B_n|} = \frac{1}{|A_n|} \ra 0,
\]
which finishes the proof. 
\end{proof}

\section{Embeddings into flabby pairs}
\label{sec:embed}
The aim of this section is to prove Theorem \ref{embed flabby}. We begin 
by recalling the notion of weak containment of unitary representations and
some of its basic properties. 

\begin{definition}[Weak containment]
Let $L$ be a countable group and let $(\cH_\pi,\pi)$ and $(\cH_\rho,\rho)$
be separable unitary representations of $L$. We say that \emph{$\pi$ is weakly contained in $\rho$}, and write $\pi \prec \rho$, if for every 
vector $v \in \cH_\pi$ and for every finite subset $F \subset L$ and $\eps > 0$, there exist vectors $w_1,\ldots,w_k \in \cH_\rho$ such that 
\[
\big| 
\langle v,\pi(s)v \rangle - \sum_{i=1}^k \langle w_i,\rho(s) w_i \rangle
\big| < \eps
\]
for all $s \in F$.
\end{definition}

In particular, if $\pi$ denotes the identity representation on $\cH_\pi = \bC$,
and $\pi$ is weakly contained in $\rho$, then there exists, for every finite 
subset $F \subset L$ and $\eps > 0$, vectors $w_1,\ldots,w_k \in \cH_\rho$
such that
\[
\big| 1 - \sum_{i=1}^k \langle w_i,\rho(s)w_i \rangle \big| < \eps,
\quad
\textrm{for all $s \in F$}.
\]
We stress that this is also the same as saying that the identity representation on any separable Hilbert space, e.g. $\ell^2(\bN)$, is weakly contained in $\rho$. 
We  note that if $L$ is a countable amenable group and $\rho$ denotes the left regular representation on $\ell^2(L)$ and $(F_n)$ is any left F\o lner sequence in $L$, then the sequence 
\[
w_n = |F_n|^{-1/2} \cdot \chi_{F_n} 
\]
in $\ell^2(L)$ can be used to show  that the identity representation is weakly contained in $\rho$. Furthermore, if $L$ is \emph{infinite}, then the Hilbert space $\ell^2(L)$ 
admits no non-zero $\rho(L)$-invariant vectors. \\

On the opposite side of amenability, we have the groups with Kazhdan's property (T), which we define here as follows.

\begin{definition}[Kazhdan's Property (T)]
A countable group $L$ has \emph{(Kazhdan's) property (T)} if whenever $(\cH,\rho)$ is 
a unitary representation of $L$ which weakly contains the identity 
representation, then the Hilbert space $\cH$ admits a non-zero $\rho(L)$-invariant vector.
\end{definition}

Prominent examples of countable groups with property (T) are lattices in 
non-compact and connected simple Lie groups with finite center real rank 
at least two, e.g. $\SL_n(\bZ)$ for $n \geq 3$. For more details about these
groups, we refer the reader to the book \cite{BV}. \\

In order to prove Theorem \ref{embed flabby}, we shall make use of Proposition H.2 in the appendix of \cite{Kech}, which asserts that if $L$ is any countable group and if $(\cH_\pi,\pi)$ and $(\cH_{\rho_o},\rho_o)$ are two \emph{infinite dimensional} unitary representations of $L$ with $\pi \prec \rho_o$, then 
for every unit vector $v \in \cH_\pi$, there exists a sequence $(\sigma_n)$ 
of unitary operators on the Hilbert space $\cH_\rho = \oplus_{\bN} \cH_{\rho_o}$ and a unit vector $w \in \cH_\rho$, such that 
\[
\lim_n \langle w, \sigma_n^{-1}\rho(s) \sigma_n w \rangle 
= \langle v, \pi(s) v \rangle, \quad \forall \, s \in L,
\]
where $\rho = \rho_o^{\oplus \bN}$ is the direct sum unitary representation 
on $\cH_\rho$. We note that if $\rho_o$
does not admit non-zero invariant vectors, then neither does $\rho$. \\

In particular, if we apply all of this to the special case when 
$\pi$ is the identity representation on $\ell^2(\bN)$ and $\pi$ 
is weakly contained in $\rho_o$, then we have proved the 
following proposition. 

\begin{proposition}[Unitary orbits and weak containment]
\label{help flabby}
Let $L$ be a countable group without property (T) and suppose that $(\cH_o,\rho_o)$ is a unitary representation of $L$, which weakly contains the identity representation. Then there exists a (possibly different) unitary representation $(\cH,\rho)$ of $L$, a unit vector $v \in \cH$ and a sequence $(\sigma_n)$ of unitary operators on $\cH$ such that
\begin{equation}
\label{conj to id}
\lim_n \langle v, \sigma_n^{-1}\rho(s)\sigma_n v \rangle = 1,
\end{equation}
for all $s \in L$. Furthermore, if $\rho_o$ does not have invariant vectors,
then neither does $\rho$.
\end{proposition}

\begin{proof}[Proof of Theorem \ref{embed flabby}]
We recall that we want to prove that whenever $L$ does not have 
Kazhdan's property (T), then $L$ embeds into a countable group $G$ which 
admits a unitary representation $(\cH,\tilde{\rho})$ with no 
non-zero $\tilde{\rho}(L)$-invariant vector, with the property that
for some unit vector $v \in \cH$ and for every sequence $(F_n)$ 
of finite subsets of $L$, there is a sequence $(s_n)$ in $G$
such that
\[
\lim_n \Big\| \frac{1}{|F_n|} \sum_{s \in F_n} \tilde{\rho}(ss_n)v \Big\| = 1.
\]
To prove this, we use the assumption that $L$ does not have property (T)
to pick a unitary representation $(\cH_o,\pi_o)$ of $L$ which
weakly contains the identity representation, but which does not 
have any non-zero invariant vectors. Let $(\cH,\rho)$ be the unitary representation of $L$ associated to $(\cH_o,\pi_o)$ via Proposition 
\ref{help flabby} and let $(\sigma_n)$ be the corresponding sequence of
unitary operators on $\cH_o$. We denote by $S$ the countable subgroup of the group of unitary operators on $\cH$ generated by the elements 
$\sigma_n$. 

We let $G$ denote the free product of $L$ and $S$ and we note that by the universal property of free products, there exists a unitary representation 
$\tilde{\rho}$ of $G$ on the Hilbert space $\cH$ which extends $\rho$ and satisfies $\tilde{\rho}(s) = s$ for all $s \in S$. Since $\rho$ does not have 
invariant vectors and $\tilde{\rho}$ extends the action of $\rho$ on the same Hilbert space, we conclude that $\tilde{\rho}$ does not admit any non-zero $L$-invariant vectors.

We now fix a sequence $(F_n)$ of finite subsets of $L$ and we note that \eqref{conj to id} guarantees that for every integer $n$, we can find an index 
$k_n$ such that  
\[
\Re \langle v, \sigma_{k_n}^{-1}\rho(s)\sigma_{k_n} v \rangle 
\geq  
1 - \frac{1}{n}
\quad 
\textrm{for all $s \in F_n^{-1} F_n$},
\]
where $\Re$ denotes the real part of a complex number. In particular, we have
\[
\Big\| \frac{1}{|F_n|} \sum_{s \in F_n} \tilde{\rho}(s\sigma_{k_n})v \Big\|^2
= 
\frac{1}{|F_n|^2} \, \sum_{(s,t) \in F_n \times F_n} \Re \langle v, \sigma_{k_n}^{-1}\rho(s^{-1}t)\sigma_{k_n} v \rangle \geq 1 - \frac{1}{n},
\]
which finishes the proof.
\end{proof}

\section{Examples of contracting triples and conjugation-thick subgroups}
\label{examplesconj}
In this section we shall give two classes of examples of (left amenable) contracting triples and thereby supplying explicit examples of triples (and groups) for which the Left Strong Ergodic Theorem fails in the strongest possible way. 

\subsection{HNN-extensions}
The first class of examples stems from special cases of HNN-extensions. 
Suppose that we are given a countable group $H$ and an 
\emph{injective} homomorphism $\alpha : H \ra H$. We define the countable group
\[
G_\alpha = \big\langle H, t \, | \, tht^{-1} = \alpha(h) \, \quad \textrm{for all $h \in H$} \big\rangle, 
\]
that is to say, $G_\alpha$ is the group generated by $H$ and the cyclic group
generated by the element $t$, which is assumed to satisfy the relation 
$tht^{-1} = \alpha(h)$ for all $h$ in $H$. We note that once $G_\alpha$ has been defined, the subgroups $H_k = t^{-k}Ht^{k}$ form an ascending 
chain of subgroups of $G_\alpha$, i.e. 
\begin{equation}
\label{asc}
\ldots < H_{-2} < H_{-1} < H < H_1 < H_2 < \ldots.
\end{equation}
If we set $L_\alpha = \cup_k H_k$, then $G_\alpha$ splits as
a semidirect product of the normal subgroup $L_\alpha$ and the cyclic 
subgroup generated by $t$. In particular, if $H$ is amenable, then so is
$G_\alpha$. \\

For instance, we could start with the group $H = \bZ$
with the injective homomorphism $\alpha(n) = pn$, where
$p$ is some non-zero integer. In this case, the group $L_\alpha$ coincides with $Z[1/p]$ and $G_\alpha$ is isomorphic
to the (outer) semi-direct product $Z[1/p] \rtimes \bZ$, where $\bZ$ acts on $Z[1/p]$ by multiplication with powers of $p$. \\

We shall now prove the following simple proposition about these groups.

\begin{proposition}
For every countable group $H$ and injective homomorphism 
$\alpha : H \ra H$, the triple $(G_\alpha,H,L_\alpha)$ is contracting.
Furthermore, if $\alpha(H) \neq H$, then $H$ 
has infinite index in $L_\alpha$.
\end{proposition}

\begin{proof}
We fix a finite subset $F \subset L_\alpha$. By the ascending chain \eqref{asc} of subgroups, there exists an index $k$ 
such that $F \subset H_k = t^{-k}Ht^{k}$, and thus $t^k F t^{-k} \subset H$,
which shows that $(G_\alpha,H,L_\alpha)$ is contracting. The second
part follows from the multiplicativity of the index; note that for every positive 
integer $k$, we have
\[
[H_k : H] = [H_k : H_{k-1}] \cdot [H_{k-1} : H ] 
= \ldots = \prod_{j=0}^{k-1} \, [H_{j+1} : H_j],
\]
and 
\[
[H_{j+1} : H_{j}] = [t^{-(j+1)}Ht^{(j+1)} : t^{-j}Ht^{j}] = [t^{-1}Ht : H ] 
= [H_1 :H],
\]
for every $j$, where the second equality is justified by the fact that conjugation by $t^{-j}$ does not change the index. In particular, we see that $H$ has infinite index in $L$ whenever $H$ is a
proper subgroup of $H_1$, which is the case if $\alpha(H) \neq H$.
\end{proof}

\subsection{Permutations with finite supports}

We shall now show that the countable group $\Sym_o(\bN)$ 
consisting of those permutations $s$ on $\bN$ for which the support 
\[
\supp(s) = \big\{ x \in \bN \, : \, s(x) \neq x \big\} 
\]
is finite, admits conjugation-thick subgroups. We note that  
$\Sym_o(\bN)$ can be written as a union of an increasing chain of finite subgroups, that is to say, it is locally finite and thus amenable. 

\begin{proposition}
For every non-empty finite set $K \subset \bN$, the subgroup 
\[
L_K = \big\{ s \in \Sym_o(\bN) \, : \, s|_{K} = id \big\} < \Sym_o(\bN).
\] 
is a proper conjugation-thick subgroup of $\Sym_o(\bN)$.
\end{proposition}

\begin{proof}
We fix a finite set $K \subset \bN$ and pick a finite subset $F \subset \Sym_o(\bN)$. 
We want to show that there exists an element $\sigma \in \Sym_o(\bN)$ such that $\sigma F \sigma^{-1} \subset L_K$, or equivalently 
$F \subset L_{\sigma^{-1}(K)}$. To prove this, we pick any finite subset 
\[
M \subset \bN \setminus \bigcup_{s \in F} \supp(s)
\]
with $|M| = |K|$ and choose an involution $\sigma$ in $\Sym_o(\bN)$ which maps $M$ onto $K$ and leaves all elements outside $M \cup K$ untouched. Clearly, $F \subset L_M = L_{\sigma^{-1}(K)}$, which 
finishes the proof.
\end{proof}

\section{Automatic pointwise ergodic theorems}
\label{sec:automatic}

We begin by describing two classes of rigid pairs of countable groups.

\begin{example}[Affine groups over countable fields]
Let $\bK$ be a countable (infinite) field and let $G_{\bK}$ denote the
semidirect product of the additive group $\bK$ and the multiplicative
group $\bK^*$, where the latter acts on $\bK$ by multiplication. 
Suppose $(\cH,\pi)$ is a unitary representation of $G_{\bK}$ without 
any $\pi(G_{\bK})$-invariant vectors, but with a $\pi(\bK^*)$-invariant 
unit vector $v_o$. We define
\[
\phi(b,a) = \langle \pi(b,a)v_o, v_o \rangle, \quad \textrm{for $b \in \bK$ and $a \in \bK^*$}, 
\]
and note that since $v_o$ is $\bK^*$-invariant, the function $\phi$ is 
independent of the second coordinate and satisfies 
\[
\phi(a \cdot b,1) = \phi(b,1) 
\quad 
\textrm{for all $b \in \bK$ and 
$a \in \bK^*$}.
\]
Since every non-zero $\bK^*$-orbit in $\bK$ equals $\bK$, we conclude that 
\[
\phi(b,a) = 
\left\{
\begin{array}{cc}
1 & \textrm{if $b = 0$} \\
\phi(1,1) & \textrm{otherwise} 
\end{array}
\right..
\]
It is now a straightforward exercise to verify that $\phi(1,1) = 0$ using 
the Left Weak Ergodic Theorem and the assumption that $(\cH,\pi)$ does 
not admit any non-zero $\pi(G)$-invariant vectors. We conclude that the 
pair $(G_{\bK},\bK^*)$ is rigid.
\end{example}

\begin{example}[Character rigid groups]
Let $G_o$ be a countable group and define the pair
\[
G = G_o \times G_o \qand H = \Delta_2(G_o),
\]
where the latter group denotes the diagonal subgroup in $G$. If 
$(\cH,\pi)$ is a unitary representation of $G$ with a $\pi(H)$-invariant
unit vector $v_o$, then one can readily verify that
\[
\phi(s,t) = \langle v_o, \pi(s,t)v_o \rangle \quad \textrm{for $(s,t) \in G$}
\]
satisfies $\phi(s,t) = \phi_o(s^{-1}t)$ for all $s, t \in G_o$, where 
$\phi_o(t) = \phi(e,t)$, and $\phi_o$ is a conjugation-invariant 
function on $G_o$ with $\phi_o(e) = 1$. Functions of this form are often
referred to as \emph{characters} in $G_o$, and we say that $G_o$ is 
\emph{character rigid} if every character which is not constant is simply
the indicator function of the identity element. Clearly, the pair $(G,H)$ 
is rigid if and only if $G_o$ is character rigid. 

Recently it has been proved
by A. Thom and J. Peterson in \cite{PT13} that the groups $\SL_n(R)$ are character rigid for $n \geq 2$ and for a wide variety of infinite rings $R$, including all infinite fields, and
A. Dudko and K. Medynets have proved that the Higman-Thompson group
is character rigid.
\end{example}

We shall now give the short proof of Theorem \ref{automatic}. 
Let $(G,H)$ is a rigid pair and suppose that $(X,\nu)$ is an 
ergodic probability measure preserving $G$-space. We note
that $G$ acts in a unitary way on $L^2_o(X,\nu)$ (equivalence 
classes of square-integrable functions with zero integral) via 
the left regular representation (Koopman representation). If 
$\varphi \in L^2_o(X,\nu)$ is $H$-invariant, then 
\begin{equation}
\label{orthonormal}
\int_X \varphi(x) \, \overline{\varphi(s^{-1} \cdot x)} \, d\nu(x) = 0
\end{equation}
for all $s \notin H$ by the assumption that $(G,H)$ is rigid. \\

Let $(F_n)$ be any strictly increasing and nested sequence of finite 
subsets of $G/H$ so that we can write
\[
F_n = \big\{s_{1}H,\ldots,s_{m_n}H\big\},
\]
for some sequence $m_n$ of integers which tend to infinity, and 
distinct elements $s_k H$ in $G/H$. 
We set
\[
\varphi_k(x) = \varphi(s_k^{-1}x) \quad \textrm{for $k \geq 1$}
\]
and note $(\varphi_k)$ is a sequence of orthogonal vectors in 
$L^2(X,\nu)$ with 
\[
\int_X |\varphi_k|^2 \, d\nu = \int_X |\varphi|^2 \, d\nu
\quad 
\textrm{for all $k$.}
\]
Theorem \ref{automatic} is now an immediate consequence of
the following rather standard lemma in measure theory (sometimes
attributed to A. Rajchman) whose proof we include here for completeness.
\begin{lemma}
Let $(X,\nu)$ be a probability measure space and suppose that
$(\varphi_{k})$ is sequence of orthonormal vectors in $L^2_o(X,\nu)$.
There exists a conull subset $X' \subset X$ such that
\[
\lim_n \frac{1}{n} \sum_{k=1}^{n} \varphi_{k}(x) = 0
\]
for all $x \in X'$.
\end{lemma}

\begin{proof}
Since 
\[
\sum_{n=1}^\infty 
\Big\| \frac{1}{n^2} \sum_{k=1}^{n^2} \varphi_{k} \Big\|^2_2 
=
\sum_{n=1}^\infty \frac{1}{n^2} < \infty,
\]
we conclude by Borel-Cantelli's Lemma, that 
\[
\lim_n \frac{1}{n^2} \sum_{k=1}^{n^2} \varphi_{k}(x) = 0
\]
for all $x$ in some $\nu$-conull subset $X_1 \subset X$. \\

Given an integer $m$, we denote by $n(m)$ the unique positive integer with 
\[
n(m)^2 \leq m \leq (n(m) + 1)^2,
\] 
so that 
\[
\frac{1}{m^2} 
\int_X \, 
\Big| 
\sum_{k=1}^{m} \varphi_{k} - \sum_{k=1}^{n(m)^2} \varphi_{k}
\Big|^2 \, d\nu
\leq 
\sum_{m=1}^\infty \frac{m-n(m)^2}{m^2} \leq \sum_{m = 1}^\infty
\big( m^{-2} + 2 \cdot m^{-\frac{3}{2}} \big) < \infty,
\]
for all $m$, since
\[
m - n(m)^2 \leq (n(m) + 1)^2 - n(m)^2 = 1 + 2n(m) \leq 1 + 2 \sqrt{m}.
\]
This shows, again by Borel-Cantelli' Lemma, that 
\[
\lim_m \frac{1}{m} \Big(\sum_{k=1}^{m} \varphi_{k}(x) - \sum_{k=1}^{n(m)^2}\varphi_k(x) \Big) = 0
\]
for all $x$ in some $\nu$-conull subset $X_2 \subset X$. Hence,
\[
\lim_m \frac{1}{m} \sum_{k=1}^m \varphi_k 
=
\lim_m 
\frac{1}{m} \cdot \Big(  \sum_{k=1}^m \varphi_k(x) - \sum_{k=1}^{n(m)^2}\varphi_k(x)  \Big) + 
\lim_m 
\frac{n(m)^2}{m} \cdot \Big(\frac{1}{n(m)^2}\sum_{k=1}^{n(m)^2}\varphi_k(x)  
\Big) = 0,
\]
for all $x \in X' = X_1 \cap X_2$, where the last equality follows from the 
fact that the ratio between the numbers $n(m)^2$ and $m$ tends to one.
\end{proof}

\section*{Appendix: Firm strong sequences in lattices in higher rank Lie groups}

We shall now derive the formulation of Proposition \ref{gn} from that of Theorem 1.7 in \cite{GN}. \\

Let $G$ be a lattice in a non-compact and connected simple Lie group
$\tilde{G}$ with trivial center. Theorem 1.7 in \cite{GN} asserts there 
exists a sequence $(F_n)$ of finite subsets of $G$ (which can be chosen 
as the intersections of $G$ with an increasing family of geometrically 
defined balls in the ambient Lie group $\tilde{G}$) with the property that 
for \emph{every} ergodic probability measure preserving $G$-space $(X,\nu)$, there exist constants $\delta > 0$ and $C$ such that
\[
\sup
\Big\{ 
\Big\| 
\frac{1}{|F_n|} \, \sum_{s \in F_n} \pi(s)f 
\Big\|_2
\, : \,
\|f\|_2 = 1
\Big\}
\leq C \cdot e^{-\delta n}, 
\]
for all $n$, where $\pi$ denotes the left regular representation (Koopman
representation) of $G$ on the Hilbert space $L_o^2(X,\nu)$. This clearly shows 
that $(F_n)$ is a firm strong sequence for left regular (Koopman) representations associated to any ergodic probability measure preserving
$G$-space. \\

In particular, if $K$ is a compact group with Haar probability measure 
$m_K$ and $\pi : G \ra K$ is a homomorphism with dense image in 
$K$, then the associated probability measure preserving $G$-action on $(K,m_K)$
given by $s \cdot k = k\pi(s^{-1})$ is ergodic, and thus
\begin{equation}
\label{ergcompact}
\lim_n \frac{1}{|F_n|} \sum_{s \in F_n} \int_K \psi(k) \, \phi(k\pi(ss_n)) \, dm_K(k)
= 
\int_K \psi \, dm_K \cdot \int_K \phi \, dm_K
\end{equation}
for all real-valued continuous functions $\psi$ and $\phi$ on $K$ and for every 
sequence $(s_n)$ in $G$. \\

We shall now show that this simple observation implies that $(F_n)$ is a 
firm strong sequence for every \emph{finite-dimensional} unitary representation of $G$. Indeed, suppose that $(\cH,\pi)$ is such a representation and let $K$ denote the \emph{compact} closure of $\pi(G)$ 
in the orthogonal  group of $\cH$. We may assume that $\cH$ does not 
admit any non-zero $\pi(G)$-invariant vectors. Since the weak topology 
and the norm topology on $\cH$ coincide, it suffices to show that
\[
\lim_n \frac{1}{|F_n|} \sum_{s \in F_n} \langle w, \pi(ss_n)v \rangle = 0
\]
for all $v$ and $w$ in a weakly dense subspace of $\cH$ and for every sequence $(s_n)$ in $G$. \\

We write 
$k \cdot v$ for the image of a vector $v$ under the action of an element $k$ in $K$, so in particular we have $\pi(s)v = \pi(s) \cdot v$ for all $s$ in $G$. We note that for every $v, w \in \cH$, the bounded function $\phi$ on $G$ defined by
\[
\phi(s) = \langle w, \pi(s) v \rangle, \quad s \in G,
\]
can be extended to a continuous function $\tilde{\phi}$ on $K$, and the 
linear space of all vectors in $\cH$ of the form
\[
w_\psi = \int_K \psi(k) \, k^{-1} \cdot w \, dm_K(k),
\]  
as $\psi$ ranges over all continuous functions on $K$ and $w$ over all 
elements in $\cH$ is weakly dense. By our observation in \eqref{ergcompact},
we have
\[
\lim_n \frac{1}{|F_n|} \sum_{s \in F_n} \langle w_\psi, \pi(ss_n)v \rangle
= 
\lim_n \frac{1}{|F_n|} \sum_{s \in F_n} \int_K \psi(k) \, \tilde{\phi}(k\pi(ss_n)) \, dm_K = \int_K \psi \, dm_K \cdot \int_K \tilde{\phi} \, dm_K,
\]
so it suffices to show that the right hand side vanishes. However, one readily
verifies that the integral of $\tilde{\phi}$ equals $\langle w, P_G v \rangle $, where $P_G$ denotes the projection onto the $\pi(G)$-invariant vectors in $\cH$,
which by assumption is trivial, and thus the right hand side vanishes. \\

If $(\cH,\pi)$ is a \emph{weakly mixing}  unitary representation of $G$, then
we can use the Gaussian measure construction (see e.g. Appendix C in \cite{Kech}) to realize every cyclic sub-representation of $(\cH,\pi)$ as a 
sub-representation of the Koopman representation associated to an 
\emph{ergodic} (for this weak mixing is crucial) probability measure preserving $G$-space, and Proposition \ref{gn} follows. \\

\end{document}